\documentclass{article}
\usepackage[latin1]{inputenc}
\usepackage{amsmath}
\usepackage{amsthm,amssymb,amsfonts}
\usepackage{graphicx}
\usepackage{verbatim}
\usepackage{color,cite}
\usepackage{latexsym}
\usepackage{lscape}
\usepackage[T1]{fontenc}
\usepackage[all,cmtip]{xy}
\usepackage{hyperref}
\usepackage{cleveref}
\usepackage{caption}
\usepackage{xspace}
\usepackage{enumitem}
\usepackage{tikz}
\usepackage{xcolor}
\usepackage{soul}

\usetikzlibrary{patterns}
\usetikzlibrary{shapes,arrows,spy,positioning}

\usetikzlibrary{decorations.pathreplacing}

\theoremstyle{plain}
\newtheorem{theorem}{Theorem}[section]
\newtheorem{lemma}[theorem]{Lemma}
\newtheorem{proposition}[theorem]{Proposition}

\newtheorem*{definition}{Definition}

\newtheorem*{remark}{Remark}

\newcommand{\nc}{\newcommand}

\nc\bB{\mathbb{B}}
\nc\bC{\mathbb{C}}
\nc\bD{\mathbb{D}}
\nc\bE{\mathbb{E}}
\nc\bF{\mathbb{F}}
\nc\bG{\mathbb{G}}
\nc\bH{\mathbb{H}}
\nc\bI{\mathbb{I}}
\nc{\bJ}{\mathbb{J}}
\nc\bK{\mathbb{K}}
\nc\bL{\mathbb{L}}
\nc\bM{\mathbb{M}}
\nc\bN{\mathbb{N}}
\nc\bO{\mathbb{O}}
\nc\bP{\mathbb{P}}
\nc\PP{\mathbb{P}}
\nc\bQ{\mathbb{Q}}
\nc\bR{\mathbb{R}}
\nc\bS{\mathbb{S}}
\nc\bT{\mathbb{T}}
\nc\bU{\mathbb{U}}
\nc\bV{\mathbb{V}}
\nc\bW{\mathbb{W}}
\nc\bY{\mathbb{Y}}
\nc\bX{\mathbb{X}}
\nc\bZ{\mathbb{Z}}

\nc\cA{\mathcal{A}}
\nc\cB{\mathcal{B}}
\nc\cC{\mathcal{C}}
\nc\cD{\mathcal{D}}
\nc\cE{\mathcal{E}}
\nc\cF{\mathcal{F}}
\nc\cG{\mathcal{G}}
\nc\cH{\mathcal{H}}
\nc\cI{\mathcal{I}}
\nc{\cJ}{\mathcal{J}}
\nc\cK{\mathcal{K}}
\nc\cM{\mathcal{M}}
\nc\cN{\mathcal{N}}
\nc\cO{\mathcal{O}}
\nc\cP{\mathcal{P}}
\nc\cQ{\mathcal{Q}}
\nc\cS{\mathcal{S}}
\nc\cT{\mathcal{T}}
\nc\cU{\mathcal{U}}
\nc\cV{\mathcal{V}}
\nc\cW{\mathcal{W}}
\nc\cY{\mathcal{Y}}
\nc\cX{\mathcal{X}}
\nc\cZ{\mathcal{Z}}

\nc\SL{\operatorname{SL}}
\nc\Ann{\operatorname{Ann}}
\nc\hol{\operatorname{hol}}

\nc\hF[1][(X,\omega)]{F^{1,0}#1}

\AtBeginDocument{\def\Re{\operatorname{Re}}}
\AtBeginDocument{\def\Im{\operatorname{Im}}}

\title{A Zero Lyapunov Exponent in Genus $3$ Implies the Eierlegende Wollmilchsau}
\author{David Aulicino\thanks{D.A. was partially supported by NSF DMS - 1738381, a grant from the Simons Foundation ($\#$853471), and several PSC-CUNY grants.}, Frederik Benirschke, and Chaya Norton\thanks{C.N. received travel support from the AMS-Simons Travel Grants, which are administered by the American Mathematical Society with support from the Simons Foundation.}}
\date{}

\begin{document}

\newcommand{\genlin}{\text{GL}_2(\mathbb{R})}
\newcommand{\splin}{\text{SL}_2(\mathbb{R})}
\newcommand{\spolin}{\text{SO}_2(\mathbb{R})}

\newcommand{\CylP}{cylinder pinch\xspace}

\newcommand{\JPD}{jump problem distance\xspace}

\maketitle

\begin{abstract}
We prove that the closed orbit of the Eierlegende Wollmilchsau is the only $\splin$-orbit closure in genus three with a zero Lyapunov exponent in its Kontsevich-Zorich spectrum.  The result recovers previous partial results in this direction by Bainbridge-Habegger-M\"oller and the first named author. \textcolor{black}{The main new contribution is the identification of the differentials in the Hodge bundle corresponding to the Forni subspace in terms of the degenerations of the surface.  We use this description of the differentials in the Forni subspace to evaluate them on absolute homology curves and apply the jump problem from the work of Hu and the third named author to the differentials near the boundary of the orbit closure.}  This results in a simple geometric criterion that excludes the existence of a Forni subspace.
\end{abstract}

\tableofcontents

\section{Introduction}

The Lyapunov exponents of the Kontsevich-Zorich cocycle provide detailed information about the straight-line flow on a translation surface, including those arising from rational billiards \cite{ZorichFlatSurfs}.  They have also played an important role in understanding the dynamics in the moduli spaces of these surfaces \cite{EskinMirzakhaniInvariantMeas}.  While generic surfaces, i.e., those with dense orbit in strata of Abelian differentials, have a simple Kontsevich-Zorich spectrum \cite{ForniDev, AvilaVianaSimp}, \cite{ForniHand} discovered a genus three translation surface, now known as the Eierlegende Wollmilchsau \cite{HerrlichSchmithusenEier} (Figure~\ref{Wollmilchsau}), with maximally many Lyapunov exponents equal to zero in its spectrum.  Surfaces with maximally many Lyapunov exponents equal to zero proved to be extremely exceptional \cite{MollerShimuraTeich, AulicinoCompDegKZ, AulicinoCompDegKZAIS, AulicinoNortonST5}; only two closed orbits in any genus have maximally many Lyapunov exponents equal to zero.  On the other hand, in genus four, for example, there are infinitely many orbits with one or more Lyapunov exponents equal to zero, or \emph{zero Lyapunov exponents} for short.

In genus two, the second Lyapunov exponent is either $1/2$ or $1/3$ depending only on the stratum in which the translation surface lies \cite[Thm.~15.1]{BainbridgeThesis}.  However, in genus three, individual Lyapunov exponents can vary depending on the orbit closure.  If a genus three translation surface has a Lyapunov exponent equal to zero, then by combining the results of \cite{AulicinoZeroExpGen3} and \cite[Prop.~4.5]{BainbridgeHabeggerMollerHNFilt}, it must generate a Teichm\"uller curve in the principal stratum, $\cH(1,1,1,1)$.  However, other than a finiteness statement using equidistribution of orbit closures following \cite{EskinMirzakhaniMohammadiOrbitClosures} combined with the result of the first named author \cite{AulicinoZeroExpGen3}, little could be said about Teichm\"uller curves with a zero Lyapunov exponent in the principal stratum.  We prove

\begin{theorem}
\label{Gen3Class}
Let $\cM$ be an orbit closure in genus three with at least one zero Lyapunov exponent in its Kontsevich-Zorich spectrum.  Then $\cM$ is the Teichm\"uller curve generated by the Eierlegende Wollmilchsau.
\end{theorem}

\begin{figure}
  \centering
  \begin{tikzpicture}[scale=.6]
    \draw (0,0) -- (8,0) -- (8,2) -- (0,2) -- (0,0);
    \draw (0,4) -- (8,4) -- (8,6) -- (0,6) -- (0,4);
    \foreach \x in {(0,0), (4,0), (0,6), (4,6), (8,0), (8,6)} {\node[circle, fill=black] at \x {};}
    \foreach \x in {(0,2),(0,4), (4,2), (4,4), (8,2), (8,4)} {\node[rectangle, fill=black] at \x {};}
    \foreach \x in {(2,4), (2,2), (6,2),(6,4)} {\node[diamond, fill=black] at \x {};}
    \foreach \x in {(2,0),(6,0), (2,6), (6,6)} {\node[star, fill=black] at \x {};}

    \node[circle, label=below:$0$] at (1,0) {};
    \node[circle, label=below:$1$] at (3,0) {};
    \node[circle, label=below:$2$] at (5,0) {};
    \node[circle, label=below:$3$] at (7,0) {};

    \node[circle, label=below:$4$] at (1,4) {};
    \node[circle, label=below:$7$] at (3,4) {};
    \node[circle, label=below:$6$] at (5,4) {};
    \node[circle, label=below:$5$] at (7,4) {};

    \node[circle, label=above:$2$] at (1,6) {};
    \node[circle, label=above:$1$] at (3,6) {};
    \node[circle, label=above:$0$] at (5,6) {};
    \node[circle, label=above:$3$] at (7,6) {};

    \node[circle, label=above:$4$] at (1,2) {};
    \node[circle, label=above:$5$] at (3,2) {};
    \node[circle, label=above:$6$] at (5,2) {};
    \node[circle, label=above:$7$] at (7,2) {};

    \draw [dashed] (2,0) -- (2,2);
    \draw [dashed] (2,4) -- (2,6);
    \draw [dashed] (4,0) -- (4,2);
    \draw [dashed] (4,4) -- (4,6);
    \draw [dashed] (6,0) -- (6,2);
    \draw [dashed] (6,4) -- (6,6);

     \end{tikzpicture}
\caption{The square-tiled surface known as the Eierlegende Wollmilchsau}
\label{Wollmilchsau}
\end{figure}
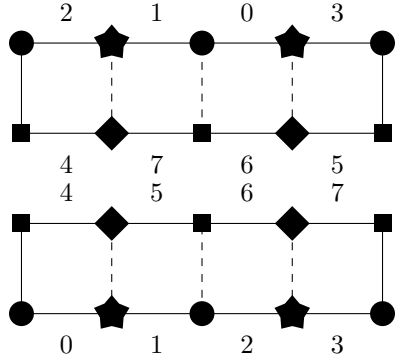

This theorem resolves a question posed by the first named author in \cite{AulicinoZeroExpGen3}.  It was proven that there are no Teichm\"uller curves with a zero Lyapunov exponent outside of the principal stratum in genus three by \cite[Prop.~4.5]{BainbridgeHabeggerMollerHNFilt}.  Using this result, \cite{AulicinoZeroExpGen3} proved that every orbit closure in genus three with a zero Lyapunov exponent must be a Teichm\"uller curve in the principal stratum, and that there are at most finitely many.  For the case of two zero Lyapunov exponents in genus three, the maximum possible, the classification of all possible orbit closures was carried out in the work of \cite{MollerShimuraTeich, AulicinoCompDegKZ, AulicinoCompDegKZAIS}.

In light of these results, it suffices to prove Theorem~\ref{Gen3Class} by focusing on Teichm\"uller curves in the principal stratum.  However, the techniques developed here are sufficiently powerful that they apply to all orbit closures and all strata in genus three.  When it is convenient to do so, we assume for some technical results that a translation surface lies on a Teichm\"uller curve or in the principal stratum.  In Appendix~\ref{GeneralizationAppendix}, we explain how to generalize the classification to arbitrary orbit closures in any stratum in genus three. 

We begin by recalling \cite[Prop.~1.1]{AulicinoZeroExpGen3}, which uses the results of \cite{FilipZeroExps} to prove that zero exponents in genus three must arise from a Forni subspace.  There are two key ingredients that facilitate the results of this paper.  First, in Section~\ref{ForniSubspaceSection}, we consider differentials in the Hodge bundle whose real parts lie in the Forni subspace.  \textcolor{black}{In \cite{AvilaEskinMollerForniBundle} the Forni subspace was defined, and its existence was established.  Given the well-known relations between real and complex absolute homology and cohomology as well as the Hodge bundle, \cite{AvilaEskinMollerForniBundle} pass freely between these spaces.  However, we disambiguate these spaces here to highlight the nuances of how each space interacts with other complementary subspaces, which embody the contributions of \cite{AulicinoZeroExpGen3, AulicinoNortonST5} and the present work.  The key insight of \cite{AulicinoZeroExpGen3} was to consider differentials in the Forni subspace and evaluate them on curves in absolute homology.  The result was stated in genus three and we include an easy generalization in Proposition~\ref{ForniSubspCrit} for future reference.  In \cite{AulicinoNortonST5}, the jump problem was applied to study the differentials in the Hodge bundle of Shimura-Teichm\"uller curves that were not in the span of the flat differential determining the translation surface.  In this case, since the derivative of the period matrix along the Teichm\"uller flow on the Teichm\"uller curve is known to be a symmetric rank one matrix, finding any non-zero term outside of the known non-zero term would produce a contradiction.  This was the key ingredient in the proof of the classification in \cite{AulicinoNortonST5}.  In the present work, we begin by considering the differentials in the Forni subspace.  By moving close to the boundary of moduli space where the surface degenerates, we are able to explicitly identify the differentials in the Forni subspace.  We then deduce how these differentials in this space interact with curves on our given translation surface, and a general statement is given in Proposition~\ref{ForniCanonicalBasisProp}.  Finally, in Proposition~\ref{NewForniCriterion}, we deduce a property that is used to exclude most cases in this work.}


The second key ingredient is the use of the solution to the jump problem \cite{HuNortonGenVarFormsAbDiffs} that played a key role in \cite{AulicinoNortonST5}.  From \cite{AulicinoZeroExpGen3}, there are six cases to consider.  Case 3 can be excluded entirely by solving the jump problem for the nodal surface in that case.  In Case 6, the jump problem is used to prove that the two cylinders are homologous and then flat techniques from \cite[$\S$5]{AulicinoNortonST5} prove that the only possible translation surface satisfying the necessary conditions is the Eierlegende Wollmilchsau.  We remark that unlike \cite[$\S$5]{AulicinoNortonST5}, no computer assistance is necessary for the proof presented here.

After setting terminology and notation in Section~\ref{PrelimSect} and recalling the solution to the jump problem from \cite{HuNortonGenVarFormsAbDiffs} in \Cref{sec:Asymp}, we introduce the holomorphic Forni subspace in Section~\ref{ForniSubspaceSection}.  We proceed in Section~\ref{ForniSubspaceSection} to recall the main technical lemma of \cite{AulicinoZeroExpGen3} and to generalize its consequences.  We then develop general results about the holomorphic Forni subspace and connect it to flat geometry and the degenerations of Riemann surfaces.  In Section~\ref{MainThmSect}, we state and prove the main theorem using the technical results in the following sections.  We also recall the six cases from \cite{AulicinoZeroExpGen3} that describe all possible cusps of a Teichm\"uller curve with a zero Lyapunov exponent.  Sections~\ref{Case124Sect} through \ref{Case6Sect} are dedicated to addressing each case and either excluding it or in Case 6, proving that the surface is in the $\splin$-orbit of the Eierlegende Wollmilchsau.  Finally, in Appendix~\ref{GeneralizationAppendix}, we explain how to generalize the result from 
Teichm\"uller curves in the principal stratum to general orbit closures in genus three.

\subsubsection*{Acknowledgments} The authors thank the Mathematical Sciences Research Institute for its hospitality during the Fall 2019 semester on Holomorphic Differentials in Mathematics and Physics.  The authors thank Matt Bainbridge, Samuel Grushevsky, and Alex Wright for helpful conversations.  \textcolor{black}{We also thank the anonymous referee for thoughtful feedback that improved the paper.}

\section{Preliminaries}
\label{PrelimSect}

\textcolor{black}{The purpose of this section is to set notation for the work.  For a more detailed introduction, we reference the reader to \cite[$\S$2]{AulicinoNortonST5}.  For background on flat surfaces and Lyapunov exponents, we refer the reader to \cite{ZorichFlatSurfs, ForniMatheusSurvey}.  For background on the jump problem and plumbing differentials, see \cite{FayThetaFcns, Yamada, HuNortonGenVarFormsAbDiffs}.}

\subsection{Flat Geometry}

A \emph{translation surface} $(X, \omega)$ is a pair consisting of a Riemann surface $X$ carrying an Abelian differential $\omega$.  If $X$ has genus $g \geq 2$, then there will be cone points with angles that are a multiples of $2\pi$ corresponding to the zeros of $\omega$.  Since the holonomy lies in $2\pi \bZ$, given a tangent vector at a point, there is a straight-line trajectory emanating from that point.  If the trajectory is closed and does not pass through a cone point, i.e., it is \emph{regular}, then there is a set of parallel trajectories homotopic to it that determine a \emph{cylinder}.  The boundaries of the cylinder necessarily consist of closed trajectories beginning and terminating at cone points.  The \emph{height} of a cylinder will always refer to the distance between its boundaries. A \emph{saddle connection} is a straight-line trajectory that begins and ends at not necessarily distinct cone points.

A direction on a translation surface is \emph{periodic} if every trajectory in that direction is closed.  This implies that there is a decomposition of the surface into cylinders in that direction.  The data consisting of the cylinders with their saddle connections and identifications between the saddle connections, but forgetting the metric data of the cylinders and saddle connections, is called a \emph{cylinder diagram}.  

In \cite{AulicinoZeroExpGen3}, a depiction of a translation surface was introduced that was well-suited to the arguments in that work.  Indeed, the convention will be useful here as well.  Typically, cylinders on a translation surfaces are drawn as parallelograms with singularities at their vertices.  We choose instead to depict the cylinders as rectangles that do not necessarily have singularities at their corners.  In order to emphasize that a single saddle connection $\sigma$ is broken by the rectangle, the left-hand portion of the saddle connection, which occurs on the right-side of the rectangle will be written $\sigma$, and the right-hand portion of the saddle connection, which occurs on the left-hand side of the rectangle will be denoted by $\sigma'$.  See Figure~\ref{SCFigConvention}.

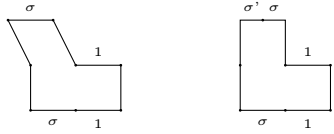
\begin{figure}[ht]
\centering
\begin{minipage}[b]{0.24\linewidth}
\begin{tikzpicture}[scale=0.30]
\draw (0,0)--(0,2)--(-1,4)--(1,4)--(2,2)--(4,2)--(4,0)--cycle;
\foreach \x in {(0,0),(0,2),(-1,4),(1,4),(2,2),(4,2),(4,0),(2,0)} \draw \x circle (1pt);
\draw(0,4) node[above] {\tiny $\sigma$};
\draw(3,2) node[above] {\tiny 1};
\draw(1,0) node[below] {\tiny $\sigma$};
\draw(3,0) node[below] {\tiny 1};
\end{tikzpicture}
\end{minipage}
\centering
\begin{minipage}[b]{0.24\linewidth}
\begin{tikzpicture}[scale=0.30]
\draw (0,0)--(0,4)--(2,4)--(2,2)--(4,2)--(4,0)--cycle;
\foreach \x in {(0,0),(0,2),(1,4),(2,2),(4,2),(4,0),(2,0)} \draw \x circle (1pt);
\draw(0.5,4) node[above] {\tiny $\sigma$'};
\draw(1.5,4) node[above] {\tiny $\sigma$};
\draw(3,2) node[above] {\tiny 1};
\draw(1,0) node[below] {\tiny $\sigma$};
\draw(3,0) node[below] {\tiny 1};
\end{tikzpicture}
\end{minipage}
 \caption{Two translation surfaces representing the same point in moduli space}
 \label{SCFigConvention}
\end{figure}

\subsection{Strata}

For each genus, the bundle of Abelian differentials over the moduli space of genus $g$ Riemann surfaces can be stratified by the orders of the zeros of the differentials in the space.  The total order of the zeros counted with multiplicity is $2g-2$.  Let $\kappa$ be a partition of $2g-2$.  Then $\cH(\kappa)$ denotes the moduli space of Riemann surfaces carrying Abelian differentials with zeros of order specified by $\kappa$.  We will use the shorthand $\cH(1^4)$ to mean $\cH(1,1,1,1)$ throughout.

These strata are not necessarily connected, but their connected components have been classified \cite{KontsevichZorichConnComps}.  We note that the stratum $\cH(1^4)$ is connected.

\subsection{Dynamics}

Strata admit a natural action by $\genlin$.  A particularly important subgroup of $\genlin$ is given by the diagonal matrices, which form a $1$-parameter family known as the \emph{Teichm\"uller geodesic flow}, or \emph{Teichm\"uller flow}.  In this manuscript, we will be concerned with a geodesic diverging to the boundary.  Define a family of matrices
$$g_t = \left(\begin{array}{cc} 1 & 0 \\ 0 & e^{t} \end{array} \right).$$
Given a horizontally periodic translation surface $(X, \omega)$, the \emph{Teichm\"uller geodesic} determined by $(X, \omega)$ is the family given by $g_t \cdot (X, \omega)$ for all $t \geq 0$.

For any translation surface we can consider the group of derivatives of affine diffeomorphisms of the surface, which naturally lie in $\genlin$.  If this group forms a lattice subgroup, then the translation surface is called a \emph{lattice surface} or \emph{Veech surface}.  Smillie proved that the $\genlin$-orbit of a translation surface is closed if and only if the translation surface is a Veech surface.  Such a closed orbit is known as a \emph{Teichm\"uller curve}.  A translation surface is called \emph{completely periodic} if the existence of a closed regular trajectory implies that every parallel trajectory is closed.  An important theorem of Veech is that Veech surfaces are completely periodic \cite{VeechTeichCurvEisen}.

\textcolor{black}{Every $\genlin$-orbit closure, after restricting to the locus of unit area translation surfaces, admits a finite $\splin$-invariant measure by \cite{EskinMirzakhaniInvariantMeas} and the Teichm\"uller geodesic flow is ergodic with respect to that measure \cite{MasFinMeasErg, VeechFinMeasErg}.}  By \cite{EskinMirzakhaniMohammadiOrbitClosures, Filip2}, every $\genlin$-orbit closure is a quasi-projective subvariety of moduli space called an \emph{invariant subvariety}.  Consider the real absolute cohomology bundle $H^1$ over the moduli space with the Gauss-Manin connection.  The \emph{Kontsevich-Zorich cocycle} (\emph{KZ-cocycle} for short) is a symplectic (orbifold) cocycle on this space that is induced by the action of the Teichm\"uller flow.  By the Oseledets multiplicative ergodic theorem, there is a well-defined set of Lyapunov exponents associated to almost every element in the moduli space.  These can be computed by considering the monodromy matrices $A_t$ given by taking longer and longer return times of the Teichm\"uller flow to a small neighborhood of $(X, \omega)$ and computing the eigenvalues of $A_t A_t^{\intercal}$, computing their logarithms, normalizing by $t$ and letting $t$ tend to infinity.  Normalizing the top Lyapunov exponent yields a symmetric set of Lyapunov exponents known as the \emph{Kontsevich-Zorich spectrum}
$$1 = \lambda_1 > \lambda_2 \geq \cdots \geq \lambda_g \geq - \lambda_g \geq \cdots \geq -\lambda_2 > -\lambda_1 = -1.$$
Due to the symmetry, we always restrict to the top $g$ Lyapunov exponents.  For a more detailed explanation of this setup and connections to flat geometry, see \cite{ZorichFlatSurfs}.

It is also possible to consider the Zariski closure of the monodromy of the KZ-cocycle.  By \cite{FilipZeroExps}, the resulting group completely determines the exact number of zero Lyapunov exponents in the KZ-spectrum.  By \cite[Prop.~1.1]{AulicinoZeroExpGen3}, the only mechanism for producing zero Lyapunov exponents in genus three is known as a Forni subspace \cite{AvilaEskinMollerForniBundle}.  The \emph{Forni subspace} is the maximal subspace of absolute (real) cohomology on which the monodromy of the KZ-cocycle restricted to this subspace is contained in a compact group.

\subsection{Degenerate Surfaces}

\textcolor{black}{Given a horizontally periodic translation surface, $g_t$ as defined above can be applied to it to increase the modulus of every cylinder.}  In this way, the modulus tends to infinity with $t$ and the core curve of every cylinder is pinched in the limit.  This results in a nodal surface in the Deligne-Mumford compactification such that the stable differential over the nodal surface has a pair of simple poles with opposite residues at each node.

There is a natural graph known as the \emph{dual graph}, or \emph{stable graph} in the literature, associated to such a stable curve.  After removing the nodes, each connected component is denoted by a vertex, and each edge corresponds to a node.  An edge can be incident with a single vertex.  Finally, each vertex is labeled with the genus of the connected component of the surface that it represents.

\section{Asymptotics of Period Matrices}\label{sec:Asymp}

\subsection{Cylinder Pinching}
\label{sec:degeneration}

Let $(X, \omega)$ be a horizontally periodic translation surface.  We call $(X', \omega')$ the \emph{cylinder pinch of $(X, \omega)$ along the family} \[\{g_t \cdot (X, \omega) = (X_t, \omega_t) \,|\, t \geq 0\}\] if $(X', \omega')$ is the limit nodal surface in the Deligne-Mumford compactification given by letting $t$ go to infinity.  We drop the family from the definition when it is not needed for the discussion at hand. In particular, the nodal Riemann surface $(X',\omega')$ is obtained by pinching the core curves of all horizontal cylinders. Thus, given a simple closed curve in $X'$, not crossing through any nodes, we can consider it as a path in $(X_t, \omega_t) = g_t \cdot (X, \omega)$.

For the rest of the section  $g'$ denotes the geometric genus of $X'$, i.e., the sum of the genera of all irreducible components. It will be convenient to use a homology basis that is adapted to the \CylP.

\begin{definition}
Let $X'$ be a cylinder pinch of $(X, \omega)$ along a family $(X_t, \omega_t)$ of geometric genus $g'$.
Let $X_v$ be the irreducible component of $X'$ associated to the vertex $v$ in the dual graph of $X'$.
We say that a homology class $[\alpha]\in H_1(X_t;\bZ)$ is supported on an irreducible component $X_v$ of $X'$ if it can be represented by a sum of simple closed curves, which are all contained in $X_v$.

We say that a symplectic homology basis 
$$\cB = \{\alpha_1, \beta_1, \ldots, \alpha_{g'}, \beta_{g'}, \ldots, \alpha_g, \beta_g\}$$
of $X_t$ is {\em adapted to a \CylP} $X'$ if the following conditions are satisfied.
\begin{enumerate}
\item \textcolor{black}{The set $\{\alpha_{g'+1},\ldots,\alpha_g\}$ is a collection of cycles on $X_t$ represented by core curves of horizontal cylinders on $X$.\footnote{If there are linear relations among the core curves in homology, then a linearly independent subset of them would be taken.}}
\item For $1\leq i\leq g'$, the classes  $\alpha_i,\beta_i$ are supported on some irreducible component of $X'$ of positive genus.
\item Furthermore, the collection of cycles $\alpha_i,\beta_i, i\leq g'$, which are supported on the irreducible component $X_v$, form a symplectic homology basis for $X_v$.
\end{enumerate}
\end{definition}

\textcolor{black}{\begin{remark}
We observe that such a basis can always be constructed by taking a symplectic basis on each component $X_v$, which will satisfy Conditions 2 and 3.  Then a collection of core curves of cylinders can be chosen to satisfy Condition 1, which will be linearly independent of the curves on each $X_v$, and they will not intersect the curves on $X_v$ because the core curves of the cylinders are homotopic to the nodes of $X'$ when $t$ tends to infinity.  Finally, we can take any completion of this basis to a symplectic basis by choosing any collection of $\beta_i$ that work.
\end{remark}}

\textcolor{black}{In general, one can only choose a homology basis locally in a neighborhood of a translation surface in moduli space.}  Here we can choose it along the the whole Teichm\"uller geodesic. In fact, it is constant in a trivialization of the bundle of relative homology.

Given the same setup as above, we let $A$ be the Lagrangian subspace spanned by $\{\alpha_1,\ldots,\alpha_g\}$ on $X$ and $B$ the complementary subspace spanned by $\{\beta_1,\ldots,\beta_g\}$. We refer to elements of $A$ and $B$ as $A$-cycles and $B$-cycles, respectively.

\subsection{Asymptotics of A-Normalized Differentials}
\label{AsymptANormDiffsSect}

\textcolor{black}{Recall that given a symplectic basis $\{\alpha_i, \beta_j\}$ of absolute homology on a Riemann surface, a basis of Abelian differentials $\{\Theta_1,\ldots,\Theta_g\}$ is  \emph{$A$-normalized} if it satisfies $\int_{\alpha_j}\Theta_i(t)=\delta_{ij}$, where $\delta_{ij}$ is the Kronecker delta.  Given a cylinder pinch $(X', \omega')$ of $(X, \omega)$ along a family $(X_t,\omega_t)$, choose a symplectic basis $\mathcal{B}$ adapted to the degeneration.  Let $\{\Theta_1(t),\ldots,\Theta_g(t)\}$ be an $A$-normalized basis on $X_t$.} 

\textcolor{black}{Now we define an $A$-normalized basis on $X'$ as follows.  This definition will depend on the choice of adapated homology basis.  For each irreducible component $X_v$ of positive genus of $X'$, choose a basis of holomorphic differentials on $X_v$ normalized against the restriction of $\{\alpha_1, \ldots, \alpha_{g'}\}$ to $X_v$.
For each vanishing cycle $\alpha_i$, for $i>g'$, we choose the unique differential having residues $\pm 1$ at the nodes crossed by $\beta_i$ with a positive residue at the preimage of the node that is reached first by $\beta_i$ relative to the orientation of $\beta_i$. 
We call the resulting basis $\{\Theta_1,\ldots,\Theta_g\}$ an $A$-normalized basis for $X'$.}

Suppose a $B$-cycle is represented by a simple closed loop $\beta_i$. Then $\Theta_i$ is supported exactly on the irreducible components of $X'$ where $\beta_i$ is supported.
Note that $\Theta_i$ has poles exactly at the nodes crossed by the $B$-cycle $\beta_i$. In particular,  $\Theta_i$ is holomorphic if and only if $i\leq g'$.

Our goal in this section is to analyze the $B$-periods $\int_{\beta_j}\Theta_i(t)$. We need to introduce some notation before we can state the results.
Recall that the nodes of the \CylP $X'$ are in correspondence with the horizontal cylinders of $X$. Denote by $C_e$ the cylinder corresponding to the node $e$, and let $\alpha_e$ be the core curve of $C_e$, which is a vanishing cycle.
We now make the assumption that the horizontal cylinders  of $(X,\omega)$ have pairwise commensurable moduli, i.e., for every edge $e$ there exists positive natural numbers $r_e\in\mathbb{N}$ such that
\begin{equation}\label{eq:moduli}
\dfrac{m(C_e)}{m(C_{e'})}=\dfrac{r_e}{r_{e'}}\text{ for } e,e'\in E(\Gamma), \quad \gcd_{e \in E(\Gamma)}(r_e)=1,
\end{equation}
where $E(\Gamma)$ is the edge set of $\Gamma$.  \textcolor{black}{We recall that the assumption that horizontal cylinders have pairwise commensurable moduli always holds for Teichm\"uller curves.  In Appendix~\ref{GeneralizationAppendix}, we will apply this to surfaces that do not necessarily have pairwise commensurable moduli by deforming them so that they do.  Such deformations always exist by the work of \cite{WrightCylDef}. }

The following observation is crucial for us to convert information about the period matrix along the geodesic flow into flat geometric information.
\begin{lemma}\label{lemma:recylinder}
Suppose $r_e=r_{e'}$, then the corresponding cylinders $C_e$ and $C_{e'}$ have the same modulus. Furthermore, if the vanishing cycles corresponding to $e$ and $e'$ are homologous, then $C_e$ and $C_{e'}$ have the same circumference and the same height.
\end{lemma}

\begin{proof}
The first claim follows from the definition of $r_e$.  The second claim follows because the vanishing cycles of $e$ and $e'$ are the core curves of $C_e$ and $C_{e'}$, respectively.  Since the period of a curve depends only on its homology class and not on a particular element in the class, the periods of the core curves of $C_e$ and $C_{e'}$ are equal, which implies that their circumferences are, too.  The heights are equal because the moduli and circumferences are equal.
\end{proof}

To analyze the behavior of periods near the nodal surface it is convenient to introduce a new coordinate
\[
s(t):= e^{-2\pi\tfrac{m(C_e)}{r_e}t},
\]
which is independent of $e$, see also \cite[Lem.~2.1]{AulicinoNortonST5}.
 Note that in particular $\lim_{t\rightarrow \infty} s=0$.
 In the sequel we are often interested in the behavior of the periods as $t$ tends to infinity, in which case it becomes more convenient to express everything in terms of the coordinate $s$.
 Depending on the circumstances we will write $\Theta(s)$ instead of $\Theta(t(s))$.

\subsection{The Solution to the Jump Problem for an Adapted Basis}

\textcolor{black}{We now prepare to apply the solution to the jump problem as developed in \cite{HuNortonGenVarFormsAbDiffs} to compute the periods of the $A$-normalized basis $\Theta_1,\ldots,\Theta_g$ along the family $X_t$.
It will be necessary to realize each surface $X_t$ as obtained by plumbing a nodal Riemann surface $Y_t$, i.e., removing small discs around the nodes of $Y_t$ and gluing the resulting boundary components.
The plumbing construction depends on a choice of local coordinates and different choices of coordinates lead to different nodal Riemann surfaces $Y_t$. For example, one can use the local coordinates introduced in \cite[Lemma 2.1]{AulicinoNortonST5}, in which case $Y_t= X'$ for all $t$.
In Case 3 below, it will be necessary to use a different coordinate system and then $Y_t$ will change with $t$. }

It is known classically that a whole neighborhood of a boundary point in the Deligne-Mumford  compactification of the moduli space of Riemann surfaces can be obtained by plumbing.
 Let $(X',\omega')$ be a cylinder pinch of $(X,\omega)$ along the family $(X_t,\omega_t)$. We use plumbing to describe the family $X_t$.
Every node $e$ of $X'$ has two preimages $q_e^{\pm}$ contained in an irreducible component $X_{v(e^{\pm})}$.
We choose a local coordinate chart $z_e^{\pm}$ at the preimage of every node.
Then there exists a family $Y_t$ of \textit{nodal} Riemann surfaces with $Y_{\infty}= X'$ such that the surface $X_t$ is obtained by removing discs $\{|z_e^{\pm}|< \sqrt{|s_e|}\}$ from \textcolor{black}{$Y_t$} and identifying the boundaries $\{|z_e^{\pm}|=\sqrt{ |s_e|}\}$ via the gluing map $z^+_e=\dfrac{s_e}{z_e^-}$, where $s_e= s_e(s)$ is a real-analytic function of $s$.
Here we considered $z_e^{\pm}$ as coordinates in $Y_t$ using local trivializations.
We write
\begin{equation}\label{eq:ae}
s_e(s)=s^{n_e}a_e(1+f_e(s)),
\end{equation}
 where $a_e\neq 0, \,n_e$ is a positive integer, and $f_e(s)=O(s)$ is real-analytic.

Although the periods of interest are independent of the choice of local coordinates used in the plumbing construction, our computation for these periods will be expressed as a series expansion whose terms depend on the local coordinates. Therefore it will be crucial for us to choose a useful coordinate system to make the computation feasible.

Let $\gamma=(e_1,\ldots,e_k)$ be a path in the dual graph $\Gamma$ of $X'$. 
We always consider the edges in a path to be oriented.
Denote the {\em weighted} length of $\gamma$ by
\[
l(\gamma)=\sum_{i=1}^k n_{e_i}.
\]

Given two $A$-normalized differentials $\Theta_i$ and $\Theta_j$ on $X_t$, we define the {\em  \JPD} $d_{\Gamma}(\Theta_i,\Theta_j)$ between $\Theta_i$ and $\Theta_j$ in the dual graph $\Gamma$
to be
\[
d_{\Gamma}(\Theta_i,\Theta_j):=\min\{ l(\gamma)\,|\,\gamma\in L(i,j)\},
\]
where $L(i,j)$ is the space of all (oriented) paths in $\Gamma$ connecting some irreducible component  where $\Theta_i$ is supported to an irreducible component where $\Theta_j$ is supported.

\begin{proposition}
\label{prop:Asymp}
Let $(X',\omega')$ be a cylinder pinch of a Teichm\"{u}ller curve and $\{\Theta_1,\dots, \Theta_g\}$ a basis of $A$-normalized differentials on $X'$. The periods $\int_{\beta j} \Theta_i(s)$ of $B$-cycles  are analytic functions of $s$ and
\[
\int_{\beta_j} \Theta_i(s):= \sum_{e\in E(\Gamma)} \langle\alpha_e,\beta_i\rangle\langle\alpha_e,\beta_j\rangle \ln(s_e)+\text{constant} + O(s^l),
\]
where $l:=d_{\Gamma}(\Theta_i,\Theta_j)$ is the \JPD between $\Theta_i$ and $\Theta_j$, and $\alpha_e$ is the vanishing cycle corresponding to the node $e$.

Here $\langle\alpha_e,\beta_j\rangle$ denotes the algebraic intersection number, computed on some surface $X_t$.
Furthermore, if $\Theta_i$ is holomorphic at every node crossed by $\beta_j$, then the logarithmic term vanishes and
\[
\lim_{s\rightarrow 0} \int_{\beta_j}\Theta_i(s)=\int_{\beta_j} \Theta_i.
\]
\end{proposition}

\begin{proof}
The result follows directly from the solution to the jump problem as developed in \cite[Thm. 4.2 + Cor. 4.6]{HuNortonGenVarFormsAbDiffs}.
\textcolor{black}{We explain how to adapt the results in (loc. cit.) to our notation.}
The {\em solution to the jump problem} as defined in (loc. cit.) is a family of differentials $\{\Theta'_1(s),\ldots,\Theta'_g(s)\}$ on $X_t$ constructed from the basis of differentials $\Theta_1,\ldots,\Theta_g$ on $X'$.
\textcolor{black}{Note that in {(loc. cit.)} the solution to the jump problem is a differential form  defined on a Riemann surface with boundary $\hat{X_t}$ such that $X_t$ is obtained by gluing the boundary components via $z\mapsto s_e/z$. Since the the solution to the jump problem agrees under the gluing, it descends to a holomorphic differential form on $X_t$.}
The choice of normalization for the Cauchy kernel used in \cite{HuNortonGenVarFormsAbDiffs}, together  with the fact that the residues of $\Theta_i$ at nodes are equal with opposite sign, implies that that solution to the jump problem yields a basis of holomorphic  differentials normalized against a basis $\cB$ adapted to the cylinder pinch.

In particular, the solution to the jump problem for $\Theta_1,\ldots,\Theta_g$ on $X_t$ agrees with $\Theta_1(s),\ldots,\Theta_g(s)$, i.e., $\Theta'_i(s)=\Theta_i(s)$ for $i=1,\ldots,g$.
We can thus apply the variational formulas \cite[Thm. 4.2 + Cor. 4.6]{HuNortonGenVarFormsAbDiffs} to compute the periods $\int_{\beta_j}\Theta_i(s)$.

In \cite[Cor.~4.6]{HuNortonGenVarFormsAbDiffs}, the logarithmic term of $\int_{\beta_j}\Theta_i(s)$ is computed to be
\[ \sum_{e\in E(\Gamma)}\langle\alpha_e,\beta_i\rangle\langle\alpha_e,\beta_j\rangle \ln(s_e).
\]
Note that  $\langle\alpha_e,\beta_i\rangle\langle\alpha_e,\beta_j\rangle$ is non-zero only if both $\beta_i$ and $\beta_j$ cross the node $e$, in which case $\Theta_i$ has a simple pole at $e$.
It then follows that the logarithmic term vanishes if $\Theta_i$ has no poles at the nodes crossed by $\beta_j$.
In that case, the constant term in \cite[Formula (4.11)]{HuNortonGenVarFormsAbDiffs} simplifies, and we conclude that the constant term of $\int_{\beta_j}\Theta_i(s)$ is equal to $\int_{\beta_j} \Theta_i$.

It remains to show that the remaining terms of the $s$-expansion of  $\int_{\beta_j}\Theta_i(s)$ are of order $O(s^l)$, where $l$ is the jump problem distance of $\Theta_i$ and $\Theta_j$.
For this we recall that in \cite{HuNortonGenVarFormsAbDiffs} the differential $\Theta_i(s)$ is written as
$\Theta_i + \sum_{k=1}^{\infty} \eta^{(k)}(s)$, where $\Theta_i$ is the differential on $X'$, which is independent of $s$, and $\eta^{(k)}(s)$ are explicitly constructed holomorphic differentials \cite{HuNortonGenVarFormsAbDiffs}.
For our purposes, it suffices to know that by \cite[Prop.~3.4]{HuNortonGenVarFormsAbDiffs}
\begin{equation}\label{eq:Paths}
\int_{\beta_j}\eta^{(k)}(s)=\sum_{\gamma\in L_k(i,j)} \left (C_{\gamma}(i,j)s^{l(\gamma)}+O(s^{l(\gamma)+1})\right),
\end{equation}
where $L_k(i,j)$ is the set of oriented paths in $\Gamma$ consisting of $k$ edges with starting point at some irreducible component supporting $\Theta_i$ and endpoint some irreducible component supporting $\beta_j$ and $C_{\gamma}(i,j)$ is a constant.
 (Note that for some $k$, $L_k(i,j)$ may be the empty set.)
By considering the lowest order term in \Cref{eq:Paths}, we conclude that $\int_{\beta_j}\sum_{k=1}^{\infty}\eta^{(k)}(s)=O(s^l)$, where $l$ is the \JPD from $\Theta_i$ to $\Theta_j$.
\end{proof}

More precise formulas for the $s$-expansion of periods were derived in \cite{HuNortonGenVarFormsAbDiffs}, which in principle facilitate the computation of periods to arbitrary precision.
 In the general formula, even the lowest order non-constant term involves multiple contributions.
 We will only need an expression for the constant $C_{\gamma}(i,j)$ in the case of paths consisting of one or two edges.

\begin{lemma}
\label{lemma:lowestOrder}Suppose $\gamma= (e)$ is a path consisting of a single oriented edge with starting point  in the irreducible component $X_{v(e+)}$ and endpoint in the component $X_{v(e-)}$.
Then 
\[
C_{\gamma}(i,j)= -a_e\hol(\Theta_i(q_e^+))\hol(\Theta_j(q_e^-)),
\]
where $\hol(\Theta_i(q_e^+))$ denotes the evaluation of the holomorphic part of $\Theta_i$ in the local coordinate charts $z_e$,\footnote{See the discussion titled ``Notation Convention'' about evaluating a differential at a point in \cite[$\S$4.4]{AulicinoNortonST5}.} and $a_e$ was defined in \Cref{eq:ae}.

If $\gamma=(e_1,e_2)$ with $e_1\neq -e_2$, then 
\[
C_{\gamma}(i,j)= a_{e_1}a_{e_2}\hol(\Theta_i(q_{e_1}^+))\hol(\Theta_j(q_{e_2}^-))\boldsymbol{\omega}_{v(e_1^-)}(q_{e_1}^-,q_{e_2}^+),
\]
where $\boldsymbol\omega_{v(e_1^-)}$ is the $A$-normalized bidifferential on the component $X_{v(e_1^-)}=X_{v(e_2^+)}$.
\end{lemma}
We refer to \cite{HuNortonGenVarFormsAbDiffs} for a precise definition of the bidifferential. For us it will be enough to know that this expression $\boldsymbol{\omega}_{v(e_1^-)}$ is a meromorphic differential on $X_t \times X_t$ and that the expression $\boldsymbol{\omega}_{v(e_1^-)}(q_{e_1}^-,q_{e_2}^+)$ is the evaluation of $\boldsymbol{\omega}$ at the points $q_{e_1}^-$ and $q_{e_2}^+$ in the chosen coordinates $z_e$.

\begin{proof}
This follows from \cite[Prop. 3.4]{HuNortonGenVarFormsAbDiffs} in the special case of a path of length one and two.
\end{proof}

We will only use the above lemma in the case where $\Theta_i$ and $\Theta_j$ are holomorphic at the nodes in question.

\subsection{Dependence on Local Coordinates}\label{sec:Coord}
In order to obtain more precise information, we will have to choose the local coordinates $z_e$ near the nodes carefully. Two specific choices will be most important for us.

 We will use the following coordinates in  our analysis of Case 6 (Section~\ref{Case6EqModSect}).
In \cite[Lem.~2.1]{AulicinoNortonST5}, local coordinates $z_e$ are constructed such that $s_e=s^{r_e}$ and such that the family $Y_t$ given above is constant and equal to $X'$.  In other words, the moduli of the nodal Riemann surfaces used in the plumbing construction remain fixed in the family.  Recall that $r_e$ was defined in \Cref{eq:moduli} and is related to the modulus of the cylinder $C_e$.
The stable differential $\omega'$ locally near the nodes is of the form $w_e\dfrac{dz_e}{z_e}$, where $w_e$ is the circumference of the cylinder $C_e$.
Notice that in (loc. cit.) vertical cylinders are used instead of horizontal cylinders. This changes the formula for $z_e$ in \cite[Lem.~2.1]{AulicinoNortonST5}, but not the formula for $s_e$.

In Case 3, it will be more convenient to use a different coordinate system. The stable curve in this case is of geometric genus one. On each genus zero component, we use the standard coordinate $z$ and let $z_e:= z-q_e$ be the coordinate centered at $q_e$, and on the elliptic component, we use any choice of local coordinates centered at the nodes.
Consequentially, we have no flat geometric interpretation of the powers $n_e$ in the expansion $s_e(s)=s^{n_e}a_e(1+f_e(s))$ anymore.
The advantage of this coordinate system is that we know the Cauchy kernel and the bidifferential on $\bP^1$ are $K_{\bP^1}(z,w)=\dfrac{1}{z-w}dz$ and $\boldsymbol{\omega}_{\bP^1}=-\dfrac{1}{(z-w)^2}dzdw$, respectively.  In particular, neither the Cauchy kernel nor the bidifferential have a holomorphic part at the origin. 
As a consequence, we obtain the following observation.
\begin{lemma}\label{lemma:noContr}
Consider the oriented path $\gamma=(e,-e)$ for some oriented edge $e$ that ends in a genus zero component.
Then $\gamma$ does not contribute to $\int_{\beta_j}\Theta_i(s)$.
\end{lemma}
In other words $\gamma$ is a path passing through a node onto a component of genus zero and immediately returns back through the same node. 

\begin{proof}
It follows from \cite[Formula (3.11)]{HuNortonGenVarFormsAbDiffs} that  a path of the form $\gamma=(e,-e)$  does not contribute to the solution to the jump problem if the holomorphic part of the Cauchy kernel is zero at the preimage of the node in the genus zero component. Since the Cauchy  kernel on $\PP^1$ is  $K_{\bP^1}(z,w)=\dfrac{1}{z-w}dz$, the holomorphic part of the Cauchy kernel is zero and the claim follows. 
\end{proof}

\section{Flat Geometry and the Forni Subspace}
\label{ForniSubspaceSection}

\subsection{Forni $B$-Matrix}

Let $x=(X,\omega)$ be a point in a stratum. For  differentials $\alpha,\beta\in H^{1,0}(X)$, the \emph{Forni $B$-form} is the bilinear form defined by
\[
B_{x}(\alpha,\beta):= \int_X \alpha\beta\dfrac{\overline{\omega}}{\omega}.
\]
We also need the real version. Let $\eta\in H^1(X,\bR)$. There exists a unique holomorphic form $h(\eta)\in H^{1,0}(X)$ with $[\Re h(\eta)]=\eta$ and the \emph{real $B^{\bR}$-form} is defined by
\[
B^{\bR}_x (\eta,\eta')= B_x(h(\eta),h(\eta'))
\]
for all $\eta,\eta'\in H^1(X,\bR)$.

\subsection{The Holomorphic Forni Subspace}

Let $\nu$ be an ergodic $\splin$-invariant measure on an invariant subvariety. The Forni subbundle $\mathcal{F}$ is the maximal $\nu$-measurable $\splin$-invariant isometric subbundle of the (real) Hodge bundle. Its fiber at a  point $x=(X,\omega)$ is the Forni subspace
\[
F(x):= \bigcap_{g\in \splin}\, g^{-1}\Ann B^{\bR}_{gx}\subseteq H^1(X,\bR).
\]
Here the annihilator of a bilinear form is
\[\Ann B^{\bR}_x:=\{ \eta \in H^{1}(X,\bR)\,|\, B^{\bR}(\eta,\eta')=0 \text{ for all } \eta'\in H^{1}(X,\bR)\}.
\]

\begin{definition}
Define the {\em holomorphic Forni subspace} to be
$$F^{1,0}(x):=\{ \omega\in H^{1,0}(X)\,|\, [\Re \omega]\in F(x)\}.$$
\end{definition}

By  \cite[Thm. 2.4]{AvilaEskinMollerForniBundle} and \cite[Lemma 3.4]{ForniMatheusZorichLyapSpectHodge}, $F(x)$ is Hodge star invariant.  Thus,
\[
F(x)\otimes_{\bR} \bC= F^{1,0}(x)\oplus \overline{F^{1,0}(x)}\subseteq H^1(X,\bC).
\]
Hence, the holomorphic Forni subspaces are fibers of a subbundle of the Hodge bundle.
Since the Forni subspace  $F(x)$ is contained in $\Ann_x^{\bR}$, it follows that
\begin{equation}\label{eq:ForniAnn}
F^{1,0}(x)\subseteq \Ann B_x.
\end{equation}

\begin{remark}\label{rem:LyapunovZero}
The Forni B-form measures the variation of the period matrix along the Teichm\"uller flow, see \cite[Lem.~2.2]{ForniMatheusZorichLyapSpectHodge}. In particular, if the Forni subspace of an orbit closure is non-trivial, in some choice of basis, the Forni B-matrix has a zero row.  Hence, the determinant of the derivative of the period matrix along the geodesic flow is constant.
\end{remark}

\subsection{The Criterion from \cite{AulicinoZeroExpGen3}}

The following lemma \cite[Lem.~4.4]{AulicinoZeroExpGen3} was proven for the real Forni subspace and can be easily adapted to the holomorphic Forni subspace.

\begin{lemma}
\label{AulForniLem}
Let $(X, \omega)$ be a translation surface with orbit closure $\cM$.  Let $C$ be a cylinder on $(X,\omega)$ with core curve $\gamma$.  If $\cM$ has nontrivial Forni subspace, then for all $\eta \in F(X, \omega)$, we have $\int_\gamma \eta = 0$, and for all $\Theta \in \hF$, we have $\int_\gamma \Theta = 0$. 
\end{lemma}

\begin{proof}
The statement for the real Forni subspace is exactly \cite[Lem.~4.4]{AulicinoZeroExpGen3}.

Since the Forni subspace $F(x)$ is Hodge star invariant, we have $[\Re \Theta],[\Im \Theta]\in F(x)$ and thus $
\int_{\gamma}\Theta = \int_{\gamma} \Re \Theta + i \int_{\gamma} \Im \Theta =0$, by
\cite[Lem.~4.4]{AulicinoZeroExpGen3}.
\end{proof}

The following proposition will not be used in this manuscript.  Nevertheless, we include it here because it generalizes \cite[Cor.~4.5]{AulicinoZeroExpGen3}.  We believe that the result here will be valuable in the study of Forni subspaces.

\begin{proposition}
\label{ForniSubspCrit}
Let $(X, \omega)$ be a genus $g$ translation surface with orbit closure $\cM$.  Let $X$ admit an absolute homology basis $\cB = \{a_1, \ldots , a_g, b_1, \ldots , b_g\}$.  Assume that for all $r$, $\{a_1, \ldots , a_r, b_1, \ldots , b_r\}$ spans a $2r$-dimensional symplectic subspace of $H_1(X, \bR)$, but we do \emph{not} assume that for any $r$, either $\{a_1, \ldots , a_r\}$ or $\{b_1, \ldots , b_r\}$ span an isotropic subspace of $H_1(X, \bR)$.  Let
\begin{itemize}
\item $\cB' = \{a_1, \ldots , a_s, b_1, \ldots , b_s, b_{s+1}, \ldots b_r \}$, where $s < r$ and $ \{b_{s+1}, \ldots, b_r\}$ spans an isotropic subspace of $H_1(X, \bR)$, or
\item $\cB' = \{b_1, \ldots, b_r\}$, spans an isotropic subspace of $H_1(X, \bR)$.
\end{itemize}
If for each $\gamma \in \cB'$, there exists $M_{\gamma} \in \cM$ such that $\gamma$ is the core curve of a cylinder on $M_{\gamma}$, then the Forni subspace of $\cM$ has dimension at most $2(g-r)$.
\end{proposition}

\begin{proof}
The Forni subspace is a symplectic subspace of $H^1(X, \bR)$ by \cite{AvilaEskinMollerForniBundle}.  Furthermore, there is a decomposition of the bundle $H^1(X, \bR)$ into the Forni bundle and its symplectic complement, which coincides with its Hodge complement.  In both cases for $\cB'$ above, the smallest symplectic subspace containing $\cB'$ has dimension at least $2r$.  Since every element of $F(X, \omega)$ evaluates to zero on $\cB'$ by Lemma~\ref{AulForniLem}, and $F(X,\omega)$ is symplectic, $\dim F(X,\omega) \leq 2g - 2r$.
\end{proof}

\subsection{The Forni and Hodge Bundles}

We now use the setup from \Cref{sec:degeneration}. 
Let $(X', \omega')$ be the cylinder pinch of $(X, \omega)$ along the family $(X_t, \omega_t)$, and let $\cB$ be a symplectic basis adapted to the cylinder pinch $(X', \omega')$ along the family $(X_t, \omega_t)$.
Let $\{\Theta_1(t),\ldots,\Theta_g(t)\}$ be an $A$-normalized basis of differentials with respect to $\cB$.

\begin{proposition}
\label{ForniCanonicalBasisProp}
Let $(X,\omega)$ be a horizontally periodic translation surface and $(X',\omega')$ be the cylinder pinch of $(X, \omega)$ along the family $(X_t, \omega_t)$ with geometric genus $g'$.
Let   
$$\cB = \{\alpha_1, \beta_1, \ldots, \alpha_{g'}, \beta_{g'}, \alpha_{g'+1}, \beta_{g'+1}, \ldots, \alpha_g, \beta_g\}$$ be a symplectic basis on $X_t$ adapted to the \CylP and $\{\Theta_1,\ldots,\Theta_g\}$ an $A$-normalized basis on $X'$.
Then  
 \[\hF[(X_t, \omega_t)]\subseteq \langle \Theta_1(t),\ldots,\Theta_{g'}(t)\rangle.
 \]
In particular, $\dim \hF[(X_t, \omega_t)] \leq g'$. Furthermore, in the case of equality,  $\dim \hF[(X_t, \omega_t)]=g'$, the following is true:
\begin{enumerate}
\item $\hF[(X_t, \omega_t)]= \langle \Theta_1(t),\ldots,\Theta_{g'}(t)\rangle$.
\item Let $(\Pi_{ij}(t))$ be the period matrix of $X_t$ normalized against the basis $\cB$. Then its derivative along the Teichm\"uller geodesic flow has the form
$$\frac{d\Pi(t)}{dt} = \left( \begin{array}{cc}0 & 0 \\ 0 & A(t) \\ \end{array} \right),$$
where $A(t)$ is the derivative of the $(g-g') \times (g-g')$-minor of the period matrix on $(X_t, \omega_t)$ restricted to $\{ \alpha_{g'+1}, \beta_{g'+1}, \ldots, \alpha_g, \beta_g\}$.
\item Each differential, $\Theta_i$ for $i\leq g'$ is supported on exactly one irreducible component of $X'$.
\end{enumerate}
\end{proposition}

Item 3 in this proposition was already noted above in Section~\ref{AsymptANormDiffsSect}, and we include it in this proposition explicitly for future reference.

\begin{proof}
Let $\{\eta_1, \ldots, \eta_{d}\}$ be a basis of $\hF[(X_t, \omega_t)]$, where $d:=\dim \hF[(X_t, \omega_t)]$. By \Cref{AulForniLem}, $\eta_i$ evaluates to 
zero on core curves of cylinders of $(X_t, \omega_t)$, which implies that for all $j > g'$, $\int_{\alpha_j}\eta_i=0$.  By the definition of an $A$-normalized basis, this implies that $\hF[(X_t, \omega_t)]$ is contained in the span of $\{\Theta_i | i \leq g'\}$. In the case $d=g'$, we have equality.

We assume $d = g'$ for the remainder of the proof. By \Cref{eq:ForniAnn} the holomorphic Forni subspace and therefore $\Theta_i$, for all $i\leq g'$, annihilates the $B$-form. 
It now follows from Rauch's variational formula along the geodesic flow \cite[Lem.~2.2]{ForniMatheusZorichLyapSpectHodge} that
the derivative of the period matrix $
\frac{d\Pi(t)}{dt}$ of the normalized basis ${\Theta_i}$ satisfies
\[
\frac{d\Pi}{dt}(t) \big\vert_{ij}=\int_{X_t} \Theta_i(t) \Theta_j(t) \frac{\bar \omega_t}{\omega_t} = 0
\]
if $i>g'$ or $j>g'$.
\end{proof}

\begin{proposition}
\label{NewForniCriterion}
Let $(X', \omega')$ be the cylinder pinch of a horizontally periodic translation surface $(X, \omega)$ along the family $(X_t, \omega_t)$ such that $X'$ has geometric genus one.  Let $\beta$ be a simple closed path on  $(X, \omega)$ such that, after applying the cylinder pinch, $\beta$ is supported on the unique elliptic component of $X'$ and that $\beta$ restricts to a path between two distinct nodes on the elliptic component.  If there exists $(Y, \eta)$ in the $\splin$-orbit closure of $(X, \omega)$ such that $\beta$ is realized as a core curve of a cylinder on $(Y, \eta)$, then the orbit closure of $(X, \omega)$ has trivial Forni subspace.
\end{proposition}

\textcolor{black}{
In \cite{AulicinoCompDegKZAIS} the  first author showed that there are only six possible cylinder diagrams for a Teichm{\"u}ller curve with zero Lyapunov exponents in genus three, which we split as Cases 1 to 6.
\Cref{NewForniCriterion} will be used below to rule out Cases 1, 2 and 4 by inspecting the cylinder diagram and finding a suitable homology class $\beta$.}

\begin{proof}
We assume by contradiction that the Forni subspace is positive dimensional.
Let   
$$\cB = \{\alpha_1, \beta_1, \ldots, \alpha_{g'}, \beta_{g'}, \alpha_{g'+1}, \beta_{g'+1}, \ldots, \alpha_g, \beta_g\}$$ be a symplectic basis on $X_t$ adapted to the \CylP and $\{\Theta_1,\ldots,\Theta_g\}$ an $A$-normalized basis on $X'$.
In particular, by \Cref{ForniCanonicalBasisProp}~(3) $\Theta_1$ is a non-zero holomorphic differential on $X'$, which is only supported on the elliptic component $E$. Since $X'$ has geometric genus one, $\hF[(X_t, \omega_t)]$ is a  $1$-dimensional complex space and is spanned by $\Theta_1(t)$, by   \Cref{ForniCanonicalBasisProp}~(1).      If $\beta$ can be realized as the core curve of a cylinder on $(X_{t_0}, \omega_{t_0})$ for some $t_0$, then $\int_\beta \Theta(t) = 0$ for all $t$ by Lemma~\ref{AulForniLem}.

By Proposition~\ref{prop:Asymp}, $\Theta_1$ is the limit of $\Theta_1(t)$ as $t$ tends to infinity. Together both facts imply that $\int_{\beta} \Theta_1=0$.
We claim that this is impossible, which will yield a contradiction.  Since $\Theta_1$ is only supported on the elliptic component $E$, the only contribution to the period comes from the restriction of $\beta$ to $E$.   By assumption, $\beta$ restricts to a path between two distinct nodes on $E$.  However, the integral of a non-zero holomorphic differential on an elliptic curve between two distinct points is never zero, so the result follows.
\end{proof}

\section{Proof of the Main Theorem}
\label{MainThmSect}

We begin by giving the proof of Theorem~\ref{MainThm} using the technical results that will be proved below.  Some of the technical results in the following sections assume that the translation surface either lies in a Teichm\"uller curve or in the principal stratum.  As mentioned in the introduction, in Appendix~\ref{GeneralizationAppendix}, we explain how to generalize these results to other orbit closures and strata in genus three, which yields a self-contained proof of Theorem~\ref{Gen3Class}.  Alternatively, Theorem~\ref{Gen3Class} can be proven by combining Theorem~\ref{MainThm} with \cite[Prop.~4.5]{BainbridgeHabeggerMollerHNFilt} and \cite{AulicinoZeroExpGen3}.

Since every orbit closure admits a horizontally periodic translation surface, it suffices to analyze every possible horizontally periodic translation surface admissible under the Forni Geometric Criterion \cite{ForniCriterion}, and these are listed in Table~\ref{MainConfigTable} and proven in \cite[Lem.~4.1]{AulicinoZeroExpGen3}.  

\begin{theorem}
\label{MainThm}
Let $\cM$ be a Teichm\"uller curve in $\cH(1^4)$ with at least one zero Lyapunov exponent in its Kontsevich-Zorich spectrum.  Then $\cM$ is generated by the Eierlegende Wollmilchsau.
\end{theorem}

\begin{proof}
 By \cite[Prop.~1.1]{AulicinoZeroExpGen3}, a zero Lyapunov exponent for a genus three translation surface can only arise from a non-trivial Forni subspace.

Since Teichm\"uller curves are generated by Veech surfaces, which are completely periodic, we consider the dual graphs that are admissible on a Veech surface with a zero Lyapunov exponent resulting from applying a cylinder pinch to a horizontally periodic translation surface.  The dual graphs were classified in \cite[Lem.~4.1]{AulicinoZeroExpGen3}: if $(X, \omega)$ is a periodic surface in a Teichm\"uller curve in $\cH(1^4)$ with zero Lyapunov exponents, then the dual graph of a cylinder pinch of $(X, \omega)$ must have one of six forms listed in Table~\ref{MainConfigTable}.

Proposition~\ref{ZeroImpCase6} proves that the cylinder diagram of a Veech surface generating a Teichm\"uller curve with a zero Lyapunov exponent must satisfy Case 6.  Proposition~\ref{Case6ImpEW} establishes that any such surface must be on the Teichm\"uller curve of the Eierlegende Wollmilchsau.
\end{proof}

Table~\ref{MainConfigTable} is a modified version of \cite[Table~1]{AulicinoZeroExpGen3}, where the second column contains the dual graph of the degenerate surface.  A key fact is that the edge of a dual graph of a cylinder pinch $(X', \omega')$ corresponds to a node, and $\omega'$ has a pair of poles at the node.  Therefore, each edge will correspond to a cylinder on a surface before a cylinder pinch is applied.

\begin{table}
$$\begin{array}{c|c}
\text{Case} & \text{Dual Graph}\\
\hline
1) & \begin{minipage}[c]{0.24\linewidth}
\centering
  \begin{tikzpicture}[scale=.5]
    \node[circle, fill=black, label=below:$1$] at (0,0) {};
    \draw [-, scale = 3] (0,0) to [out=-45,in=45, loop] (0,0);
    \draw [-, scale = 3] (0,0) to [out=135,in=225, loop] (0,0);
     \end{tikzpicture}
\end{minipage} \\
2) & \begin{minipage}[c]{0.24\linewidth}
\centering
  \begin{tikzpicture}[scale=.5]
    \node[circle, fill=black, label=below:$0$] at (0,0) {};
    \node[circle, fill=black, label=below:$1$] at (3,0) {};
    \draw [-] (0,0) to [out=45,in=135] (3,0);
    \draw [-] (0,0) to [out=0,in=180] (3,0);
    \draw [-] (0,0) to [out=-45,in=225] (3,0);
     \end{tikzpicture}
\end{minipage}  \\
3) & \begin{minipage}[c]{0.24\linewidth}
\centering
  \begin{tikzpicture}[scale=.5]
    \node[circle, fill=black, label=below:$0$] at (0,0) {};
    \node[circle, fill=black, label=below:$1$] at (2,0) {};
    \draw [-, scale = 3] (0,0) to [out=135,in=225, loop] (0,0);
    \draw [-] (0,0) to [out=45,in=135] (2,0);
    \draw [-] (0,0) to [out=-45,in=225] (2,0);
     \end{tikzpicture}
\end{minipage} \\
4) & \begin{minipage}[c]{0.24\linewidth}
\centering
  \begin{tikzpicture}[scale=.5]
    \node[circle, fill=black, label=below left:$0$] at (0,0) {};
    \node[circle, fill=black, label=below right:$0$] at (2,0) {};
    \node[circle, fill=black, label=below:$1$] at (1,-1.3) {};
    \draw [-] (0,0) to [out=45,in=135] (2,0);
    \draw [-] (0,0) to [out=-45,in=225] (2,0);
    \draw [-] (0,0) to [out=-60,in=135] (1,-1.3);
    \draw [-] (2,0) to [out=-120,in=45] (1,-1.3);
     \end{tikzpicture}
\end{minipage} \\
5) & \begin{minipage}[c]{0.24\linewidth}
\centering
  \begin{tikzpicture}[scale=.5]
    \node[circle, fill=black, label=below:$2$] at (0,0) {};
    \draw [-, scale = 3] (0,0) to [out=-45,in=45, loop] (0,0);
     \end{tikzpicture}
\end{minipage} \\
6) & \begin{minipage}[c]{0.24\linewidth}
\centering
  \begin{tikzpicture}[scale=.5]
    \node[circle, fill=black, label=below:$1$] at (0,0) {};
    \node[circle, fill=black, label=below:$1$] at (3,0) {};
    \draw [-] (0,0) to [out=45,in=135] (3,0);
    \draw [-] (0,0) to [out=-45,in=225] (3,0);
     \end{tikzpicture}
\end{minipage}  \\
\end{array}$$
\caption{A complete list of dual graphs for surfaces that permit a zero Lyapunov exponent in genus three.}
\label{MainConfigTable}
\end{table}

We will adopt the following convention.  When we say that a translation surface satisfies one of the cases, then it will be implicit that the surface is horizontally periodic and satisfies that case.

\begin{definition}
A cylinder is \emph{simple} if each of its boundaries consists of a single saddle connection.
\end{definition}

In order to prove Theorem~\ref{MainThm}, we prove that no translation surface generating a Teichm\"uller curve with a zero Lyapunov exponent can decompose into cylinders satisfying any of Cases 1 through 5.  The exclusion of each of these cases is the subject of the following sections.  In the Appendix~\ref{GeneralizationAppendix}, we also treat the case of invariant subvarieties in genus three that are not Teichm\"uller curves.  Hence, we will state the following result in a more general context.

\begin{proposition}
\label{ZeroImpCase6}
Let $(X, \omega) \in \cH(1^4)$ be completely periodic with a non-trivial Forni subspace.  Then any decomposition of $(X, \omega)$ into cylinders satisfies Case 6.
\end{proposition}

\begin{proof}
By \cite[Lem.~4.1]{AulicinoZeroExpGen3}, applying a cylinder pinch to a periodic translation surface with non-trivial Forni subspace results in a nodal surface with dual graph listed in Table~\ref{MainConfigTable}.  By Propositions~\ref{Case1Excluded}, \ref{Case2Excluded}, \ref{Case3Exc} and \ref{Case4AExc}, no cylinder decomposition of $(X, \omega)$ can satisfy Cases 1, 2, 3, or 4, respectively.  It is easy to see that Case 5 always has a transverse direction with a simple cylinder (see \cite[Lem. 4.1]{AulicinoNortonST5}).  Such a surface must decompose into cylinders by complete periodicity.  Observe that neither of the cylinders in Case 6 are simple, that Cases 1 through 4 are the only cases where a cylinder can be simple, and none of them are possible if $(X, \omega)$ has non-trivial Forni subspace.  It follows that $(X, \omega)$ cannot decompose into cylinders satisfying any of the Cases 1 through 5.
\end{proof}

\section{Cases 1, 2 and 4}
\label{Case124Sect}

\subsection{Case 1}
\label{Case1Sect}

\begin{proposition}
\label{Case1Excluded}
If a genus three translation surface satisfies Case 1, then its  orbit closure has trivial Forni subspace.
\end{proposition}

\begin{proof}
Let $(X, \omega)$ be a translation surface satisfying Case 1.  Then $(X,\omega)$ decomposes into two cylinders.  By \cite[Lem. 4.2]{AulicinoZeroExpGen3}, since the core curves of the two cylinders are not homologous, there exists a saddle connection $\sigma$ on both sides of one of the cylinders.  Thus, there are straight-line trajectories from $\sigma$ to itself determining a simple cylinder that only crosses one of the cylinders.  Let $\beta$ denote the core curve of the simple cylinder.

Let $(X', \omega')$ be the cylinder pinch of $(X, \omega)$.  After removing the nodes, $\beta$ is a path between two distinct punctures on an elliptic curve.  By Proposition~\ref{NewForniCriterion}, the orbit closure of $(X, \omega)$ has trivial Forni subspace.
\end{proof}

\subsection{Case 2}

\begin{proposition}
\label{Case2Excluded}
If a genus three translation surface satisfies Case 2, then its  orbit closure has trivial Forni subspace.
\end{proposition}

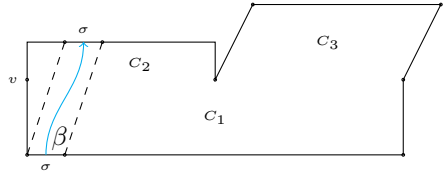
\begin{figure}[ht]
\centering
\begin{minipage}{0.5\linewidth}
\begin{tikzpicture}[scale=0.5]

\draw (0,0)--(0,3)--(5,3)--(5,2)--(6,4)--(11,4)--(10,2)--(10,0)--(2,0)--cycle;
\foreach \x in {(1,3),(2,3),(10,2),(0,2),(5,2),(0,0),(1,0),(10,0)} \draw \x circle (1pt);
\foreach \x in {(6,4),(11,4)} \draw \x circle (1pt);

\draw(5,1) node[] {\tiny $C_1$};
\draw(3,2.4) node[] {\tiny $C_2$};
\draw(8,3) node[] {\tiny $C_3$};

\draw(1.5,3) node[above] {\tiny $\sigma$};


\draw(0.5,0) node[below] {\tiny $\sigma$};
\draw(0,2) node[left] {\tiny $v$};

\draw [dashed] (0,0) -- (1,3);
\draw [dashed] (1,0) -- (2,3);
\draw [->, cyan] (.5,0) to [out=90,in=-90] (1.5,3);
\node [right] at (.4,.45) {$\beta$};
\end{tikzpicture}
\end{minipage}
 \caption{A depiction of a translation surface in Case 2 that shows the existence of a cylinder with core curve $\beta$ supported on an elliptic curve and a sphere after pinching the core curves of every horizontal cylinder.}
 \label{Config2NonSimpLem}
\end{figure}

\begin{proof}
Rotate to the horizontal direction for convenience of the figures.  As explained in description of Configuration~2 in \cite[$\S$4.1]{AulicinoZeroExpGen3}, there is a unique way of identifying three cylinders with parallel core curves so that there is a simple zero between them, e.g., the $3$-cylinder diagram in $\cH(1,1)$.  This identification is depicted between the bottoms of cylinders $C_2$ and $C_3$ and the top of cylinder $C_1$ in Figure~\ref{Config2NonSimpLem}.  We claim that there always exists a closed trajectory transverse to the horizontal direction that crosses $C_1$ and $C_2$ exactly once and determines a cylinder with core curve $\beta$ as depicted in Figure~\ref{Config2NonSimpLem}.  We emphasize that we only use cutting and gluing of the translation surface (which preserves the point in moduli space) and $\splin$ (which preserves its orbit closure) to deform a given $(X, \omega)$ into the arrangement depicted in Figure~\ref{Config2NonSimpLem}.

Draw $C_1$ so that any saddle connection $\sigma$ is located on its bottom as in the figure.  Without loss of generality, let $C_2$ be the cylinder with $\sigma$ on its top.  Shear $(X, \omega)$ by the horocycle flow so that the bottom left-hand corner of $C_2$ is located at the point $v$ in the figure.  Cut and glue the cylinder $C_2$ so that it follows the convention described in Figure~\ref{SCFigConvention}.  Given any regular point on $\sigma$ on the bottom of $C_1$, there is a closed trajectory from that point to its copy on the top of $C_2$.  Indeed the fact that $\sigma$ is contained in the top of $C_2$ implies that the circumference of $C_2$ is greater than or equal to the length of $\sigma$, so such a trajectory does exist.  Let $\beta$ denote this trajectory.  Its maximal homotopy class relative to the singularities of $\omega$ determine the desired cylinder.

Let $(X', \omega')$ be the cylinder pinch of $(X, \omega)$.  After removing the nodes, $\beta$ restricts to a path between two distinct punctures on the elliptic curve.  By Proposition~\ref{NewForniCriterion}, the orbit closure of $(X, \omega)$ has trivial Forni subspace.
\end{proof}

\subsection{Case 4}
\label{Case4Sect}

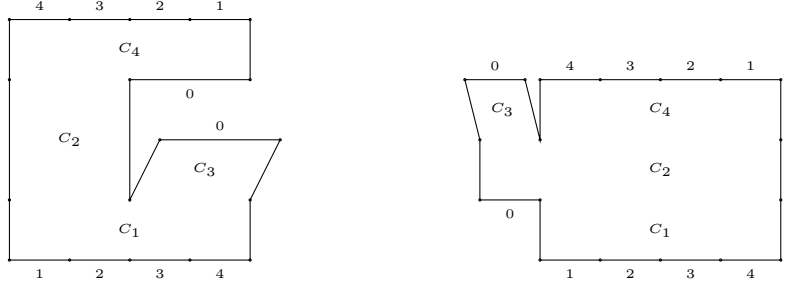
\begin{figure}[ht]
\centering
\begin{minipage}[b]{0.24\linewidth}
\centering
\begin{tikzpicture}[scale=0.40]
\draw (0,0)--(0,8)--(8,8)--(8,6)--(4,6)--(4,2)--(5,4)--(9,4)--(8,2)--(8,0)--cycle;
\foreach \x in {(0,2),(8,2),(0,6),(4,6),(8,6),(4,2),(5,4),(9,4)} \draw \x circle (1pt);
\foreach \x in {(0,0),(2,0),(4,0),(6,0),(8,0),(0,8),(2,8),(4,8),(6,8),(8,8)} \draw \x circle (1pt);

\draw(1,8) node[above] {\tiny $4$};
\draw(3,8) node[above] {\tiny $3$};
\draw(5,8) node[above] {\tiny $2$};
\draw(7,8) node[above] {\tiny $1$};
\draw(1,0) node[below] {\tiny $1$};
\draw(3,0) node[below] {\tiny $2$};
\draw(5,0) node[below] {\tiny $3$};
\draw(7,0) node[below] {\tiny $4$};

\draw(6,6) node[below] {\tiny $0$};
\draw(7,4) node[above] {\tiny $0$};

\draw(4,1) node {\tiny $C_1$};
\draw(2,4) node {\tiny $C_2$};
\draw(6.5,3) node {\tiny $C_3$};
\draw(4,7) node {\tiny $C_4$};
\end{tikzpicture}
\end{minipage}
\begin{minipage}[t]{0.24\linewidth}
\centering
\begin{tikzpicture}[scale=0.40]
\end{tikzpicture}
\end{minipage}
\begin{minipage}[b]{0.24\linewidth}
\centering
\begin{tikzpicture}[scale=0.40]
\draw (0,0)--(0,2)--(-0.5,4)--(1.5,4)--(2,2)--(2,4)--(10,4)--(10,-2)--(2,-2)--(2,0)--cycle;
\foreach \x in {(0,2),(2,2),(10,2),(10,0),(0,0),(2,0),(2,-2),(4,-2),(6,-2),(8,-2),(10,-2), (-0.5,4),(1.5,4),(2,4),(4,4),(6,4),(8,4),(10,4)} \draw \x circle (1pt);
\draw(0.5,4) node[above] {\tiny $0$};
\draw(3,4) node[above] {\tiny $4$};
\draw(5,4) node[above] {\tiny $3$};
\draw(7,4) node[above] {\tiny $2$};
\draw(9,4) node[above] {\tiny $1$};
\draw(1,0) node[below] {\tiny $0$};
\draw(3,-2) node[below] {\tiny $1$};
\draw(5,-2) node[below] {\tiny $2$};
\draw(7,-2) node[below] {\tiny $3$};
\draw(9,-2) node[below] {\tiny $4$};

\draw(6,-1) node {\tiny $C_1$};
\draw(6,1) node {\tiny $C_2$};
\draw(.75,3) node {\tiny $C_3$};
\draw(6,3) node {\tiny $C_4$};
\end{tikzpicture}
\end{minipage}
 \caption{Cylinder Diagrams 4A (left) and 4B (right)}
 \label{Case4CylDiagsFig}
\end{figure}

By \cite[Lem.~4.10]{AulicinoZeroExpGen3}, there are two $4$-cylinder diagrams satisfying Case 4, which we call 4A and 4B following \cite{AulicinoZeroExpGen3}, and they are depicted in Figure~\ref{Case4CylDiagsFig}.  It was proven in \cite[Lem.~4.12]{AulicinoZeroExpGen3} that Cylinder Diagram 4B is impossible in a rank one invariant subvariety with a non-trivial Forni subspace.  We use Proposition~\ref{NewForniCriterion} to give a simpler proof that excludes Cylinder Diagram 4B.

\begin{proposition}
\label{Case4AExc}
If a genus three translation surface $(X, \omega) \in \cH(1^4)$ satisfies Case 4, then its  orbit closure has trivial Forni subspace.
\end{proposition}

\begin{proof}
By \cite[Lem.~4.10]{AulicinoZeroExpGen3}, there are two $4$-cylinder diagrams satisfying Case 4, and they are depicted in Figure~\ref{Case4CylDiagsFig}.  
We claim that regardless of whether the translation surface satisfies Cylinder Diagram 4A or 4B, there exists a cylinder $C$ crossing each cylinder $C_1$, $C_2$, and $C_4$ exactly once before closing. In Cylinder Diagram 4A, this is the content of Proposition~\ref{CylExists}, proved in the next section.

For Cylinder Diagram 4B, cylinder $C$ is the non-horizontal cylinder with core curve $\beta$ depicted in Figure~\ref{Case4BPfFig}.  The existence of cylinder $C$ can be seen as follows.  Consider cylinder $C_2$ and cut and glue it so that saddle connection $0$ lies on the bottom of cylinder $C_2$ exactly where it does.  Then shear the surface with an element of $\splin$ so that the bottom of cylinder $C_3$ lies directly over saddle connection $0$.  Consequentially, the bottom of cylinder $C_4$ lies directly over the top of cylinder $C_1$.  Cut and reglue cylinders $C_1$ and $C_4$ so that they they are rectangles following the convention of Figure~\ref{SCFigConvention}.  From Figure~\ref{Case4CylDiagsFig} we see that every saddle connection on the top of $C_4$ is contained in the bottom of $C_1$.  Therefore, there is a regular point on the top of $C_4$ that is identified to a point on the bottom of $C_1$ and by considering its homotopy class, it determines a cylinder $C$ with core curve $\beta$ as depicted in Figure~\ref{Case4BPfFig}.

Let $(X', \omega')$ be the cylinder pinch of $(X, \omega)$.  After removing the nodes, $\beta$ restricts to a path between the two distinct punctures on the resulting elliptic curve.  By Proposition~\ref{NewForniCriterion}, the orbit closure of $(X, \omega)$ has trivial Forni subspace.
\end{proof}

\begin{figure}[ht]
\centering
\begin{minipage}[b]{0.5\linewidth}
\begin{tikzpicture}[scale=0.40]
\draw (0,0)--(0,2)--(-0.5,6)--(1.5,6)--(2,2)--(2,4)--(10,4)--(10,-2)--(2,-2)--(2,0)--cycle;
\foreach \x in {(0,2),(2,2),(10,2),(0,0),(2,0),(10,0),(-0.5,6),(1.5,6),(7,4),(9,4),(8,-2),(6,-2)} \draw \x circle (1pt);
\draw[dashed] (8,-2)--(9,4);
\draw[dashed] (6,-2)--(7,4);
\draw(4,1) node {\tiny $C_2$};
\draw(4,-1) node {\tiny $C_1$};
\draw(0.75,4) node {\tiny $C_3$};
\draw(4,3) node {\tiny $C_4$};
\draw(0.5,6) node[above] {\tiny 0};
\draw(8,4) node[above] {\tiny $\sigma$};
\draw(1,0) node[below] {\tiny 0};
\draw(7,-2) node[below] {\tiny $\sigma$};

\draw [->, cyan] (7,-2) to [out=90,in=-90] (8,4);
\node [above left] at (7.75,.85) {$\beta$};
\end{tikzpicture}
\end{minipage}
 \caption{A non-horizontal cylinder $C$ in Cylinder Diagram 4B with core curve $\beta$}
 \label{Case4BPfFig}
\end{figure}
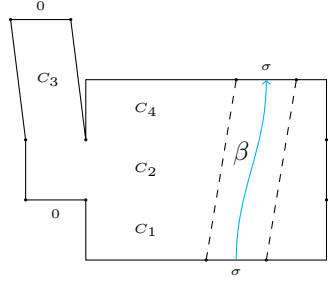

To finish the proof of \Cref{Case4AExc}, it remains to treat the Cylinder Diagram 4A.

\begin{proposition}
\label{CylExists}
Let $(X, \omega) \in \cH(1^4)$ be a translation surface satisfying Cylinder Diagram 4A.  Then there exists a cylinder with core curve crossing $C_1$, $C_2$ and $C_4$ exactly once.
\end{proposition}

\begin{proof}
We transform a translation surface $(X, \omega)$ satisfying Cylinder Diagram 4A as follows so that it is depicted as in Figure~\ref{Case4AFigCylPf}.  Let $C_2$ be the cylinder with the larger circumference of the cylinders $\{C_2, C_3\}$.  If $C_1$ (and necessarily $C_4$) have unit circumference, then $C_2$ necessarily has circumference $s \geq \frac{1}{2}$.  Shear the surface so that $C_2$ is depicted as a rectangle with its singularities at its corners.  Then the cylinders $C_1$ and $C_4$ should be placed below and above $C_2$, respectively, as rectangles.  We do \emph{not} assume that the top corners of $C_4$ and the bottom corners of $C_1$ are singularities following the convention of Figure~\ref{SCFigConvention}.

\begin{figure}[ht]
\centering
\begin{minipage}[b]{0.24\linewidth}
\begin{tikzpicture}[scale=0.40]
\draw (0,0)--(0,8)--(8,8)--(8,7)--(4,7)--(4,2)--(8,2)--(8,0)--cycle;
\foreach \x in {(0,2),(0,7),(4,7),(4,2)} \draw \x circle (1pt);

\foreach \x in {(0,0),(3,0),(4, 0), (7,0), (8,0),(0,8),(1,8),(4,8),(5,8),(8,8)} \draw \x circle (1pt);
\draw[dashed] (3,0)--(4,8);
\draw[dashed] (3.1,0)--(4.1,8);
\draw(1,4.5) node {\tiny $C_2$};
\draw(1,7.25) node {\tiny $C_4$};
\draw(1,1) node {\tiny $C_1$};
\draw(0.5,8) node[above] {\tiny 3};
\draw(2.5,8) node[above] {\tiny 2};
\draw(4.5,8) node[above] {\tiny 1};
\draw(6.5,8) node[above] {\tiny 0};
\draw(1.5,0) node[below] {\tiny 0};
\draw(3.5,0) node[below] {\tiny 1};
\draw(5.5,0) node[below] {\tiny 2};
\draw(7.5,0) node[below] {\tiny 3};

\end{tikzpicture}
\end{minipage}
\centering
\begin{minipage}[b]{0.24\linewidth}
\begin{tikzpicture}[scale=0.20]
\end{tikzpicture}
\end{minipage}
\centering
\begin{minipage}[b]{0.24\linewidth}
\begin{tikzpicture}[scale=0.40]
\draw (0,0)--(0,8)--(8,8)--(8,7)--(4,7)--(4,2)--(8,2)--(8,0)--cycle;
\foreach \x in {(0,2),(0,7),(4,7),(4,2)} \draw \x circle (1pt);

\foreach \x in {(4.5,0),(5.5,0),(6, 0), (6.75,0),(0.5,8),(1.25,8),(1.75,8),(2.75,8)} \draw \x circle (1pt);
\draw[dashed] (0,0)--(4,8);
\draw[dashed] (.5,0)--(4.5,8);
\draw(1,4.5) node {\tiny $C_2$};
\draw(1,7.25) node {\tiny $C_4$};
\draw(4,1) node {\tiny $C_1$};
\draw(0.25,8) node[above] {\tiny 0'};
\draw(0.875,8) node[above] {\tiny 3};
\draw(1.5,8) node[above] {\tiny 2};
\draw(2.25,8) node[above] {\tiny 1};
\draw(5.125,8) node[above] {\tiny 0};
\draw(2.25,0) node[below] {\tiny 0'};
\draw(5,0) node[below] {\tiny 1};
\draw(5.75,0) node[below] {\tiny 2};
\draw(6.375,0) node[below] {\tiny 3};
\draw(7.375,0) node[below] {\tiny 0};

\end{tikzpicture}
\end{minipage}
 \caption{Finding cylinders as described in the proof of Proposition~\ref{CylExists}}
 \label{Case4AFigCylPf}
\end{figure}
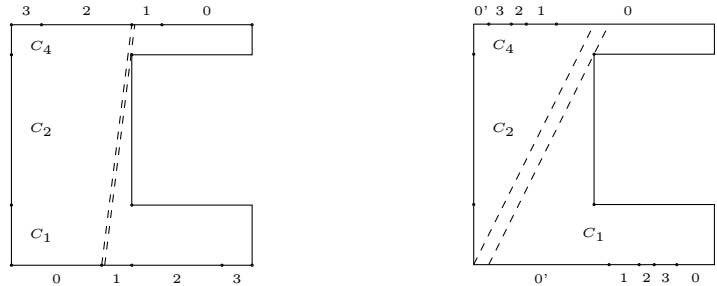

Identify the bottom of $C_1$ and the top of $C_4$ in Figure~\ref{Case4AFigCylPf} with the interval $[0, 1]$.  For this proof, we distinguish every point on the bottom of $C_1$, and the top of $C_4$, even though some are identified, e.g., $0$ and $1$.  Let $f:[0,1]\to [0,1]$ be the piecewise isometry that is continuous on half-open intervals that are open on the right and describes how the bottom of $C_1$ is glued to the top of $C_4$ such as in the examples depicted in Figure~\ref{Case4AFigCylPf}.\footnote{In fact, $f$ is an interval exchange transformation that is pre and post-composed with a translation.}  Let $\mu$ denote the Lebesgue measure on $\bR$. Observe that for any measurable set $J$, $\mu(J) = \mu(f(J))$.

Consider the interval $J=[0, s]$ on the bottom of $C_1$.  If there exists an interval $(a, b) \subset J$ on the bottom of $C_1$ such that its copy $f((a,b))$ lies in the interval $[0, s]$ on the top of $C_4$, then the trajectories from $(a, b)$ to itself form a cylinder as desired.  

We claim that if $s > \frac{1}{2}$, then we are done. Since $f$ preserves the Lebesgue measure, we have  $\mu(J) = \mu(f(J))>\frac{1}{2}$. Thus $\mu(J \cap f(J)) > 0$, which implies that there exists some subinterval $(a,b) \subset J \cap f(J)$ as above.

It remains to examine the case $s=\frac{1}{2}$. In this case, $J = [0, \frac{1}{2}]$ on the bottom of $C_1$.  If any positive measure portion of $J$ occurs in $(0, \frac{1}{2})$ on the top of $C_4$, then we are done as above.  Hence, $f(J)$ is a subset of $[\frac{1}{2}, 1]$.

There exists a unique $x\in [0,1)$ such that $f(x)=1/2$. Then $x$ has to be contained in $[0,1/2)$, since otherwise some positive measure set $(x,x+\epsilon)$ is mapped into $[1/2,1]$ which contradicts that $\mu(f(J))=1/2$. Now that we know that $x\in [0,1/2)$, we can connect the interval $(x,x+\epsilon)$ on the bottom of $C_1$ to $(1/2,1/2+\epsilon)$ on the top of $C_4$ using straight lines to construct a cylinder that passes through $C_1, C_2$ and $C_4$ exactly once because $C_4$ has positive height.  Two such examples are depicted in Figure~\ref{Case4AFigCylPf}.
\end{proof}

\section{Case 3}

\begin{figure}
\begin{centering}
\includegraphics[scale=0.3]{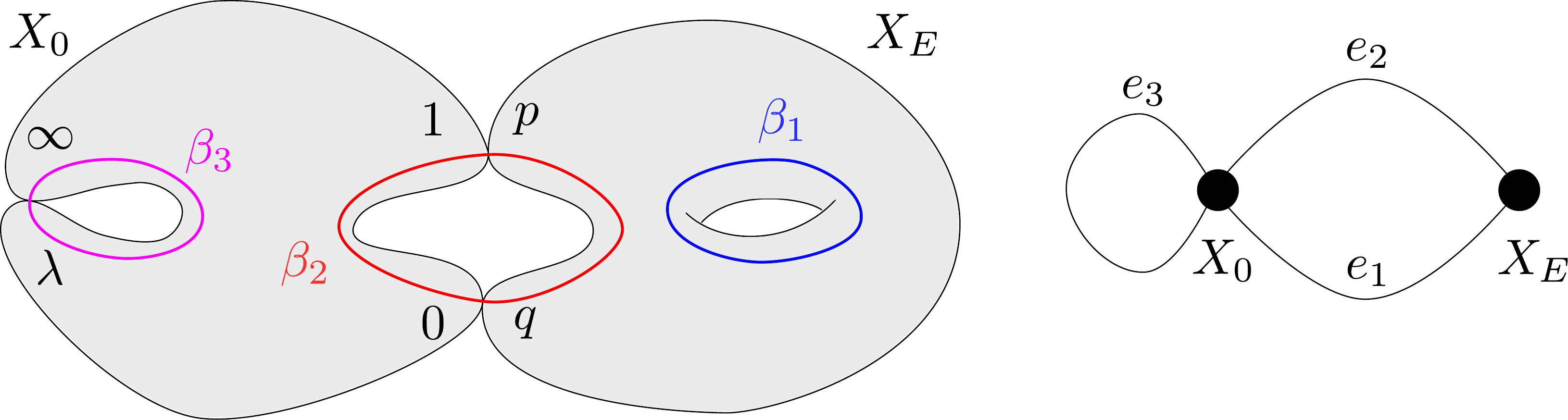}
\caption{The nodal surface in Case 3 and its dual graph}
\label{fig:Case3}
\end{centering}
\end{figure}

\begin{proposition}
\label{Case3Exc}
If a genus three translation surface satisfies Case 3, then its  orbit closure has trivial Forni subspace.
\end{proposition}

Let $(X',\omega')$ be the cylinder pinch  of $(X,\omega)$.
In Case 3, as seen in \Cref{fig:Case3}, the nodal curve $X'$ has two irreducible components: one component $X_0$ of genus zero and one elliptic component $X_E$.
The component $X_0$ has four marked points, which we choose to be $0,1,\lambda,\infty$. We identify $\lambda$ and $\infty$ to form a self-node $e_3$ and identify $0$ and $1$ with points  $p,q$ on $X_E$ to form two nodes  $e_1,e_2$ connecting $X_0$ and $X_E$.

We now choose a symplectic basis adapted to the \CylP $(X', \omega')$ above such that the $B$-cycles are of the following form.
The cycle $\beta_1$ is supported on the elliptic curve $X_E$, the cycle $\beta_2$ is a loop passing from  $X_0$ to the elliptic curve $X_E$ and back, not passing through the self node, and $\beta_3$ is a cycle through the self node, only supported on the $X_0$.
Let $\{\Theta_1,\Theta_2,\Theta_3\}$ be the corresponding $A$-normalized basis.

\begin{proof}[Proof of Prop.~\ref{Case3Exc}]
We assume by contradiction that the Forni subspace of $(X, \omega)$ is non-trivial. Since $X'$ has geometric genus one, it follows from  Proposition~\ref{ForniCanonicalBasisProp} that the holomorphic Forni subspace is $1$-dimensional and generated by the differential $\Theta_1(s)$ supported on the elliptic curve. Furthermore, the first row and column of the period matrix with respect to the chosen symplectic basis is constant along the geodesic flow.

As stated in \Cref{sec:Coord}, to evaluate the periods along the geodesic flow, we need to choose a local coordinate chart near every node. In this case, it is convenient to use the standard coordinate $z$ on $\PP^1$ \textcolor{black}{and the global holomorphic coordinate on the elliptic curve $E$, which by abuse of notation, we also denote by $z$}. We can still describe the geodesic flow by removing a small disc of radius $\{|z_e|\leq \sqrt{|s_e|}\}$, where $s_e=s^{n_e}a_e(1+f_e(s))$ by \Cref{eq:ae}.  We will write $n_i$ and $a_i$ instead of $n_{e_i}$ and $a_{e_i}$, respectively, for the remainder of the proof.
\textcolor{black}{Let $e_1$ be the node between $0$ and $q$, and let $e_2$ be the node between $1$ and $p$ (see Figure~\ref{fig:Case3}).}

We first claim that $n_{1}=n_{2}$.  \textcolor{black}{
To see this we consider the period \[
\int_{\beta_3}\Theta_1(s) = \int_{\beta_1}\Theta_3(s).\]
Since $\Theta_1$ spans the Forni subspace, the period $\int_{\beta_3}\Theta_1(s)$ is constant.}
Assume for the sake of contradiction that $n_{1}<n_{2}$. Then it follows from the asymptotic formula \Cref{lemma:lowestOrder} that
\[\int_{\beta_1}\Theta_3(s)= a_{1}\Theta_3(0)\Theta_1(q)s^{n_{1}}+ O(s^{n_{1}+1}),\]
which is nonzero since both  differentials $\Theta_1$ and $\Theta_3$ have no zero.  This yields a contradiction because $n_{1}$ is a positive integer and $\int_{\beta_1}\Theta_3(s)$ is constant.  Therefore, $n_{1} \geq n_{2}$, and by the symmetry of this argument, $n_{1} \leq n_{2}$ as well.

We now focus on \textcolor{black}{$\int_{\beta_1} \Theta_1(s)$.} It follows from $n_{1}=n_{2}$ that \[\int_{\beta_1} \Theta_1(s)=O(s^{2n_{1}})\] and that there are exactly  four oriented loops of weighted length $2n_1$ starting in the elliptic component. All remaining such loops  contain one of these four paths and thus have larger weighted length. 
The four oriented paths of weighted length $2n_{1}$ starting in the elliptic component can be described as follows. Let $e_1$ and $e_2$ be the two oriented nodes connecting the elliptic component to the genus zero components, oriented such that the start point is on the elliptic component (see Figure~\ref{fig:Case3}). The four paths $\gamma_i$, for $i\in\{1,2,3,4\}$ consist of all oriented loops, starting on the elliptic component and consisting of exactly two edges.
In other words,
\[
\gamma_1=(e_1,-e_1),\gamma_2=(e_2,-e_2), \gamma_3=(e_1,-e_2), \gamma_4=(e_2,-e_1).
\]
The contributions from $\gamma_1$ and $\gamma_2$ are zero by \Cref{lemma:noContr}.

Thus, only $\gamma_3$ and $\gamma_4$ can contribute.
It follows from the symmetry of the normalized bidifferential $\dfrac{1}{(z-w)^2}dwdz$ that both paths have the same contribution.
\textcolor{black}{Therefore, we have the following expansion  by \Cref{lemma:lowestOrder}
\begin{equation}\label{eq:NonVan}
\int_{\beta_1} \Theta_1(s)=\text{constant} + 2a_{1}a_{2}s^{n_{1}+n_{2}}\omega_{X_E(s)}(p)\omega_{X_E(s)}(q)\dfrac{1}{(1-0)^2} + O(s^{n_{1}+n_{2}+1}).
\end{equation} 
Since a holomorphic differential on $X_E(s)$ has no zeros, independent of $s$, we conclude that the entry $\int_{\beta_1} \Theta_1$ in the period matrix is not constant.
This contradicts  that $\Theta_1$ lies in the Forni subspace.} We remark that in our chosen coordinate system for plumbing, the moduli of the elliptic curve varies with $s$.  Therefore, in \Cref{eq:NonVan} the evaluation of the holomorphic differential $\omega_{X_E}(s)$ on the elliptic curve $X_E(s)$ appears.

\end{proof}

\textcolor{black}{
\begin{remark}
As with Cases 1, 2 and 4 above, 
 we believe that Proposition~\ref{NewForniCriterion} can be used to rule out Case 3 as well.  However, in \cite{AulicinoCompDegKZAIS}, Case 3 was split into three subcases.  One of them (Case 3B), is easy to exclude with Proposition~\ref{NewForniCriterion}.  However, it is not clear to the authors how the other two cases can be addressed without careful geometric arguments and dividing Cases 3A and 3C into more subcases.  For this reason, we feel it is cleaner and more efficient to use the uniform approach to Case 3 that was given above.
\end{remark}
}

\section{Case 6}
\label{Case6Sect}

Here we use the jump problem to prove that the two cylinders in Case 6, which have homologous core curves, have equal moduli.  It follows that the interiors of the cylinders are isometric, which allows us to use the arguments from \cite[$\S$5]{AulicinoNortonST5} to conclude the proof of \Cref{MainThm}.  We will prove below that in $\cH(1^4)$, there is a unique cylinder diagram satisfying Case 6 (see Figure~\ref{Wollmilchsau}).

\subsection{Equal Moduli}
\label{Case6EqModSect}

\begin{proposition}
\label{Case6EqMod}
Let $(X, \omega)$ be a translation surface with non-trivial Forni subspace.  Let $(X, \omega)$ decompose into cylinders satisfying Case 6 such that the ratio of the moduli of the two cylinders is rational.  Then the moduli of the two cylinders are equal.  Furthermore, the periods and heights of the cylinders are equal.
\end{proposition}

Consider the cylinder pinch $(X', \omega')$ of a horizontally periodic translation surface $(X, \omega)$.  In this case the nodal curve $X'$ has two irreducible components. Both of these are elliptic curves, which we denote by $E_1$ and $E_2$. They are joined at two nodes $p_1\sim p_2$ and $q_1\sim q_2$. 
Choose a symplectic basis adapted to $X'$ such that the $B$-cycles
are of the following form.
The loop $\beta_1$ is a cycle supported on $E_1$, $\beta_2$ is a cycle supported on $E_2$, and $\beta_3$ is the cycle corresponding to the loop in the dual graph of $X'$.

\begin{proof}
By contradiction and without loss of generality, assume $r_1<r_2$, in which case, by \Cref{prop:Asymp}, the lowest order terms in the period matrix are as follows

$$\Pi(s)=\begin{pmatrix}
\hbox{const.}+O(s^{2r_1})&  O(s^{r_1}) & \hbox{const.}+O(s^{r_1}) \\
  O(s^{r_1})& \hbox{const.}+O(s^{2r_1}) & \hbox{const.}+O(s^{r_1})\\
\hbox{const.}+O(s^{r_1})& \hbox{const.}+O(s^{r_1})&\ln(s^{r_1})+\ln(s^{r_2}))
\end{pmatrix}$$
We recall the exponents $r_{e_i}$ from \Cref{eq:moduli}, which are related to the ratios of the cylinders on $(X,\omega)$. In the above formula, we write $r_i$ instead of $r_{e_i}$.
\textcolor{black}{Since $(X,\omega)$ has a non-trivial Forni subspace, the derivative of the period matrix has zero determinant along the geodesic flow. We will obtain a contradiction by computing the lowest order term of the determinant using the jump problem.}

We denote the entries in the derivative of the period matrix by \[
\dfrac{d \Pi(s)}{d s}=(\pi'_{ij}(s))_{ij}.
\] It follows from the existence of a Forni subspace that $\dfrac{d\Pi(s)}{d s}$ has zero determinant.  (See the remark at the end of \Cref{rem:LyapunovZero} or Proposition~\ref{ForniCanonicalBasisProp} Part 2.)

By expanding the determinant, \textcolor{black}{which we already know must be zero}, we see that
\[
0 = \det\left(\dfrac{d \Pi(s)}{d s}\right)=\dfrac{(r_1+r_2)(\pi'_{12})^2(s)}{s}+ O(s^{2r_1-2}).
\]
\textcolor{black}{Now the goal is to show that $(\pi_{12}'(s))^2/s$ has order exactly $2r_1 - 3$.  The next calculation will compute the coefficient of the term of order $2r_1 - 2$ in $(\pi_{12}'(s))^2$.  Therefore, after dividing by $s$, this term has order $2r_1 - 3$, and then we will prove that the coefficient of this term is non-zero to reach a contradiction.}

Since $r_1<r_2$, there exists a unique oriented path whose weighted length equals to the jump problem distance of $\Theta_1$ to $\Theta_2$. The path is given by the single oriented edge $e_1$. 
Thus, by \Cref{lemma:lowestOrder} and \Cref{prop:Asymp}, we conclude
$$
\pi'_{12}(s)=-r_1s^{r_1-1}\Theta_{1}(p_1)\Theta_2(p_2) +O(s^{r_1}),
$$
which is not identically zero because a holomorphic differential on an elliptic curve is nowhere zero. 
\textcolor{black}{Therefore, the derivative of the period matrix  does not have zero determinant.} This yields a contradiction and implies $r_1=r_2$.

Finally, from \Cref{lemma:recylinder}, the circumferences and heights of the cylinders are equal.
\end{proof}

\subsection{Reduction to the Wollmilchsau}

Before proceeding, we remark that the method of proof for the following proposition is identical to the one used in \cite{AulicinoNortonST5}.  In \cite{AulicinoNortonST5}, the problem was reduced to a large finite problem and a computer search was implemented.   \textcolor{black}{Here a computer assisted proof can be avoided.} 

\begin{proposition}
\label{Case6ImpEW}
Let $(X, \omega) \in \cH(1^4)$ be a completely periodic translation surface with non-trivial Forni subspace.  Let $(X, \omega)$ decompose into cylinders satisfying Case 6 such that the ratio of their moduli is rational. Then $(X,\omega)$ generates the Teichm\"uller curve of the Eierlegende Wollmilchsau.
\end{proposition}

\begin{proof}
We claim that there is a unique cylinder diagram in $\cH(1^4)$ satisfying Case 6.  Since the core curves of the cylinders are homologous, the cylinders have equal circumference.  Denote the cylinders by $C_1$ and $C_2$ as in Figure~\ref{WindowLemmaCorFig}.  Observe that all of the saddle connections on the bottom of $C_1$ are identified to the saddle connections on the top of $C_2$ and vice versa.  Similarly, all of the saddle connections on the bottom of $C_2$ are identified to the saddle connections on the top of $C_1$ and vice versa.  Consider the operation of cutting the core curves of each cylinder and gluing the top half of $C_1$ to the bottom half of $C_2$ and the bottom half of $C_1$ to the top half of $C_2$.  This results in two $1$-cylinder surfaces, each of which are contained in $\cH(1,1)$.  We leave the reader to check that there is a unique $1$-cylinder diagram in $\cH(1,1)$.\footnote{It is worth noting for the appendix that there is also a unique $1$-cylinder diagram in $\cH(2)$.}  Since both of these $1$-cylinder diagrams are unique, if we reverse the cutting operation above and revert to the original translation surface in $\cH(1^4)$, we conclude that that cylinder diagram is also unique.

By Proposition~\ref{Case6EqMod}, the heights of both cylinders equal.  We cut and glue and, if necessary, deform a translation surface satisfying Case 6 as follows so that it appears as depicted in Figure~\ref{WindowLemmaCorFig}.  Let $\tau_0$ and $\sigma_0$, with lengths $t_0$ and $s_0$, respectively, be the longest saddle connections on the bottoms of $C_1$ and $C_2$, respectively.  Cut and glue $C_1$ if necessary so that $\tau_0$ appears on the bottom of $C_1$ as in Figure~\ref{WindowLemmaCorFig}.  Shear the surface so that $\sigma_0$ lies directly above $\tau_0$ as in Figure~\ref{WindowLemmaCorFig}.  Finally, cut and glue $C_2$ using the convention of Figure~\ref{SCFigConvention} so that $\sigma_0$ lies on the bottom of $C_2$ as in Figure~\ref{WindowLemmaCorFig}.

We normalize the circumference of the cylinders to $1$ for convenience.  Without loss of generality, let $t_0 \geq s_0$.  Since there are four saddle connections in the boundary of each cylinder, it follows that $t_0 \geq s_0 \geq \frac{1}{4}$.

Observe that every cylinder has exactly four saddle connections on each of its boundaries.  By complete periodicity, every closed trajectory determines a cylinder decomposition, and by Proposition~\ref{ZeroImpCase6}, that cylinder decomposition must satisfy Case 6.  Therefore, each boundary of every cylinder in any direction must contain four saddle connections.  On the other hand, if a non-horizontal straight-line trajectory crosses each of the cylinders in Figure~\ref{WindowLemmaCorFig} exactly once before closing, then each of its boundaries would consist of at most two saddle connections.  Therefore, the saddle connection $\tau_0$ cannot intersect the region of length $2s_0$ on the top of $C_2$ in Figure~\ref{WindowLemmaCorFig} because it would imply the existence of a closed trajectory crossing $C_1$ and $C_2$ exactly once.

We claim that the quantities $s_0$, $t_0$, and $t_{start}$ are subject to the constraints derived in \cite[$\S$5]{AulicinoNortonST5}, which are summarized in \cite[Cor.~5.10]{AulicinoNortonST5}.  To see this we refer to \cite[Fig.~2]{AulicinoNortonST5}, which has been reproduced as Figure~\ref{WindowLemmaCorFig}.  As noted above, there cannot exist a regular trajectory crossing each of the cylinders exactly once before closing.  We claim that such a trajectory must exist if any regular point in $\tau_0$ on the bottom of $C_1$ occurs in one of the two intervals\footnote{In fact, it is a single interval because the vertical sides of the rectangle are identified to form a cylinder.} bounded by black squares in Figure~\ref{WindowLemmaCorFig} of lengths $2s_0$ and $t_0$.  Indeed, the dashed lines show the boundaries of these regions, which are drawn using the fact that the heights of $C_1$ and $C_2$ are equal.  The reader can check that if $\tau_0$ intersected either of these intervals bounded by black squares, then there would exist a regular closed trajectory from $\tau_0$ on the bottom of $C_1$ passing through $\sigma_0$ and closing when it reached the top of $C_2$.  Hence, the interior of $\tau_0$ cannot intersect either of these intervals, and this implies the inequalities
$$1 - 2t_0 - 2s_0 \geq t_{start} \geq 0,$$
cf. \cite[Cor.~5.10]{AulicinoNortonST5}, where the circumference $2d_{opt}$ can be replaced with $1$ to obtain the inequality above.

Ignoring the middle term in the inequalities, the assumption that $t_0 \geq s_0 \geq \frac{1}{4}$ implies that the inequality is only satisfied exactly when
$$t_0 = s_0 = \frac{1}{4}.$$
Since these saddle connections were assumed to be the largest on their side of the cylinder, all of the saddle connection lengths are exactly equal to $\frac{1}{4}$ because there is a unique partition of $1$ into four real numbers such that the largest number is $\frac{1}{4}$.  Hence, every saddle connection has equal length.  Finally, $t_{start} = 0$ because
$$1 - 2t_0 - 2s_0 = 0 \geq t_{start} \geq 0.$$

Having determined that every saddle connection has equal length and the location of one of the saddle connections on each side of each cylinder, the fact that there is a unique cylinder diagram satisfying Case 6 in $\cH(1^4)$ implies that the translation surface is exactly the Eierlegende Wollmilchsau as depicted in Figure~\ref{Wollmilchsau}.
\end{proof}

\begin{figure}
  \centering
  \begin{tikzpicture}[scale=.9]
    \draw (0,0) -- (12,0) -- (12,1) -- (0,1) -- (0,0);
    \draw (0,3) -- (12,3) -- (12,4) -- (0,4) -- (0,3);
    \draw [dashed] (0,0) -- (2,1);
    \draw [dashed] (2,3) -- (4,4);
    \draw [dashed] (3,0) -- (0,1);
    \draw [dashed] (12,3) -- (9,4);
    \foreach \x in {(0,0), (3,0), (0,1), (2,1), (12,0), (12,1)} {\node[circle, fill=black] at \x {};}
    \foreach \x in {(0,3), (2,3), (12,3), (5,4), (8,4)} {\node[circle, fill=black] at \x {};}
    \foreach \x in {(0,4), (4,4), (9,4), (12,4)} {\node[rectangle, fill=black] at \x {};}
    \node[circle, label=above:$\tau_0$] at (13/2,4) {};
    \node[circle, label=below:$\tau_0$] at (3/2,0) {};
    \node[circle, label=above:$\sigma_0$] at (1,1) {};
    \node[circle, label=below:$\sigma_0$] at (1,3) {};
    \node[circle, label=left:$C_1$] at (0,1/2) {};
    \node[circle, label=left:$C_2$] at (0,3+1/2) {};
    \draw [decoration={brace,mirror,amplitude=7}, decorate] (0,-.5) --node[below=3mm]{$t_0$} (3,-.5);
    \draw [decoration={brace,mirror,amplitude=7}, decorate] (0,2.5) --node[below=3mm]{$s_0$} (2,2.5);
    \draw [decoration={brace,amplitude=7}, decorate] (0,4.5) --node[above=3mm]{$2s_0$} (4,4.5);
    \draw [decoration={brace,amplitude=7}, decorate] (4,4.5) --node[above=3mm]{$t_{start}$} (5,4.5);
    \draw [decoration={brace,amplitude=7}, decorate] (9,4.5) --node[above=3mm]{$t_0$} (12,4.5);
    \draw [decoration={brace,amplitude=7}, decorate] (5,4.5) --node[above=3mm]{$t_0$} (8,4.5);
     \end{tikzpicture}
\caption{Coordinates specifying the location of $\tau_0$ on the top of $C_2$.  Circles correspond to cone points and squares may or may not be cone points (Reproduced from \cite[Fig.~2]{AulicinoNortonST5})}
\label{WindowLemmaCorFig}
\end{figure}
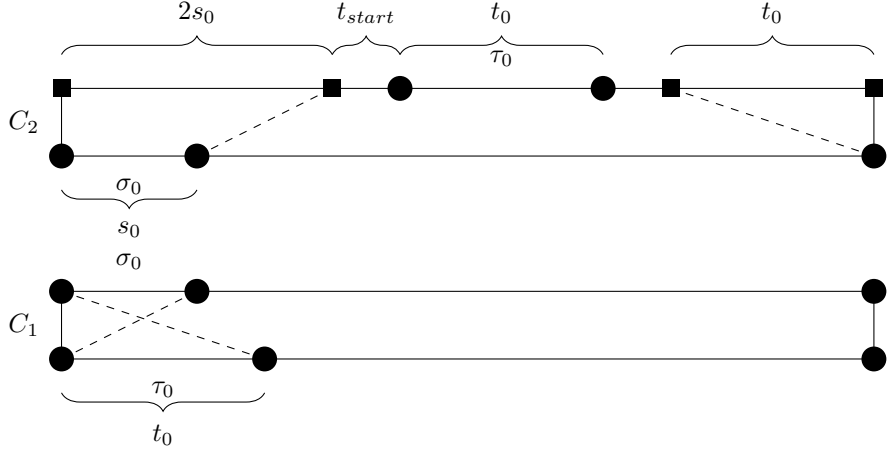

\appendix

\section{Orbit Closures and Other Strata}
\label{GeneralizationAppendix}

The results above were stated in their greatest possible generality in the context of $\splin$-orbit closures in genus three.  Nevertheless, occasionally it was necessary to assume that the translation surface possessed complete periodicity, was in the principal stratum, or had rational ratios of moduli of parallel cylinders.  We explain here how to extend Theorem~\ref{MainThm} to Theorem~\ref{Gen3Class}.  In this appendix, we assume familiarity with some properties of invariant subvarieties, e.g., field of definition and rank \cite{WrightFieldofDef, WrightCylDef}.

\subsection{Other Strata}

The arguments in Case 1, 2, 3 and 5 apply to all orbit closures in genus three with non-trivial Forni subspace. Only in Case 4 and 6 do we need to modify the arguments slightly.

\begin{proposition}
\label{Case4Genus3}
If a genus three translation surface satisfies Case 4, then it has trivial Forni subspace.
\end{proposition}

\begin{proof}
In Case 4, the translation surface must lie in $\cH(2,1,1)$ if it does not lie in the principal stratum (see \cite[$\S$4.6]{AulicinoZeroExpGen3}).  The result follows from the observation that no proof in Section~\ref{Case4Sect} required saddle connection $3$ in Figures~\ref{Case4CylDiagsFig} and \ref{Case4AFigCylPf} to have positive length.  Indeed, letting saddle connection $3$ have length zero yields the two cylinder diagrams in $\cH(2,1,1)$ and the reader can verify that the proofs still hold.
\end{proof}

\begin{proposition}
\label{Case6Genus3Teich}
Let $(X, \omega)$ be a completely periodic genus three translation surface with non-trivial Forni subspace.
If $(X, \omega)$ satisfies Case 6 and the ratio of the moduli of the cylinders is rational, then $(X,\omega)$ generates the Teichm\"uller curve of the Eierlegende Wollmilchsau.
\end{proposition}

\begin{proof}
Since the principal stratum was already addressed in Proposition~\ref{Case6ImpEW}, it suffices to focus on the remaining strata.  We claim that in Case 6, the surface must lie in $\cH(2,1,1)$ or $\cH(2,2)$.  This can be seen because the total order of the zeros between the two cylinders must be exactly two.  This follows because after a cylinder pinch, we see two elliptic curves, each with two simple poles, and the total order of the zeros and poles on an elliptic curve is zero.

By contradiction, assume that there is such a translation surface satisfying Case 6 and the assumptions of this proposition outside of the principal stratum.  Then by Propositions~\ref{ZeroImpCase6} and \ref{Case4Genus3}, every periodic direction on the surface must satisfy Case 6.  If there is a double zero between two of the cylinders, then there are three saddle connections between them.  We claim that there always exists a cylinder transverse to the horizontal direction that crosses each cylinder exactly once, which would contradict the fact that every periodic direction has to satisfy Case 6 as in the proof of Proposition~\ref{Case6ImpEW}.  In the proof of Proposition~\ref{Case6ImpEW}, we had $t_0 \geq \frac{1}{4}$ and $s_0 \geq \frac{1}{4}$.  With three saddle connections in at least one of the boundaries, we can conclude that at least one of $t_0$ and $s_0$ is greater than or equal to $\frac{1}{3}$, in which case we get
$$0 \leq t_{start} \leq 1 - 2s_0 - 2t_0 \leq 1 - 2\cdot\frac{1}{3} - 2\cdot\frac{1}{4} = \frac{-1}{6}.$$
This contradiction proves the non-existence of Teichm\"uller curves with non-trivial Forni subspace outside of the principal stratum.
\end{proof}

\subsection{Invariant Subvarieties}
\label{app:Comp}

\begin{proof}[Proof of Thm.~\ref{Gen3Class}]
We focus on the rank of the $\splin$-orbit closure in the sense of \cite{WrightCylDef}.  First, rank three orbit closures have a trivial Forni subspace by \cite{AvilaEskinMollerForniBundle}.

If an $\splin$-orbit closure has rank one, then it is completely periodic by \cite[Thm.~1.5]{WrightCylDef}.  By Propositions~\ref{ZeroImpCase6} and \ref{Case4Genus3}, a non-trivial Forni subspace and complete periodicity imply that every periodic direction must satisfy Case 6.  If the ratio of the moduli of the two cylinders in Case 6 is rational, then Propositions~\ref{Case6ImpEW} and \ref{Case6Genus3Teich} apply and we conclude.  If the ratio is irrational, then the translation surface cannot generate a Teichm\"uller curve.  By \cite[Thm.~1.9]{WrightCylDef}, such an orbit closure must be arithmetic because the ratio of circumferences of parallel cylinders in Case 6 is one.  By contradiction, assume that such an orbit closure exists.  Then it contains infinitely many (arithmetic) Teichm\"uller curves with a positive dimensional Forni subspace by \cite[Lem.~2.2]{AulicinoZeroExpGen3}.  However, Propositions~\ref{Case6ImpEW} and \ref{Case6Genus3Teich} prove that only one such Teichm\"uller curve exists in genus three, which is a contradiction that proves that the only rank one orbit closure with a non-trivial Forni subspace in genus three is the Teichm\"uller curve of the Eierlegende Wollmilchsau.

Similarly, rank two orbit closures in genus three are arithmetic and therefore, they too must contain infinitely many Teichm\"uller curves.  If a rank two orbit closure had a non-trivial Forni subspace, then all of the Teichm\"uller curves contained in it would as well by \cite[Lem.~2.2]{AulicinoZeroExpGen3}, which again is impossible.
\end{proof}

\bibliography{fullbibliotex}{}

\end{document}